\documentclass[11pt,reqno]{amsart}
\usepackage{amsmath,amssymb,amsthm,amsfonts,latexsym, amsthm, amscd, mathrsfs, stmaryrd}
\usepackage[colorlinks=true,linkcolor=blue,citecolor=blue]{hyperref}
\usepackage{cite}
\usepackage{graphicx,color,colortbl}
\usepackage{upgreek}
\usepackage[mathscr]{euscript}
\usepackage{hhline}
\numberwithin{equation}{section}
\usepackage{dynkin-diagrams}
\usepackage{epic}
\allowdisplaybreaks
\numberwithin{equation}{section}
\usepackage{float}
\usepackage{tikz}
\usetikzlibrary{
    cd,
	shapes,
	arrows,
	positioning,
	decorations.markings,
	decorations.pathmorphing,
	circuits.logic.US,
	circuits.logic.IEC,
	fit,
	calc,
	plotmarks,
	matrix
}
\usepackage{todonotes,cancel}
\usepackage{enumerate}

\newtheorem{proposition}{Proposition}[section]
\newtheorem{lemma}[proposition]{Lemma}

\newtheorem{theorem}[proposition]{Theorem}

\newtheorem{thmintro}{Theorem}

\theoremstyle{definition}

\newtheorem{example}[proposition]{Example}

\theoremstyle{remark}
\newtheorem{remark}[proposition]{Remark}

\newcommand{\Rmnum}[1]{\expandafter\@slowromancap\romannumeral #1@}
\def \g{\mathfrak{g}}

\def \N{\mathbb{N}}
\def \Q{\mathbb{C}}
\def \C{\mathbb{C}}
\def \Z{\mathbb{Z}}
\def \I{\mathbb{I}}

\def \bF{\mathbb{F}}

\def \bc {\mathbf{c}}
\def \bd {\mathbf{d}}
\def \fg{\mathfrak{g}}
\def \A{\mathcal{A}}

\def \wI{\I_{\circ}}
\def \wItau{\I_{\circ,\tau}}
\def \bI{\I_{\bullet}}

\def \gr{\mathrm{gr}}

\def \cR{\mathcal{R}}
\def \cP{\mathcal{P}}

\def \bs{\mathbf{r}} 
\def \bF{\mathbb{F}}
\def \bw{w_\bullet}
\def \bwi{w_{\bullet,i}}

\def \bbw{{\boldsymbol{w}}}

\def \ba{\mathbf{a}}

\def \P{\mathcal{P}}
\def \k{\kappa}

\def \bT{\mathbf{T}}

\def \T{\mathrm{T}}
\def \Id{\mathrm{Id}}

\newcommand{\U}{\mathbf{U}}
\newcommand{\tK}{K}

\newcommand{\qbinom}[2]{\begin{bmatrix} #1\\#2 \end{bmatrix} }

\def \ov{\overline}
\def \un{\underline}

\def \X{\mathcal{X}}

\newcommand{\nc}{\newcommand}
\nc{\greentext}[1]{\textcolor{green}{#1}}
\nc{\redtext}[1]{\textcolor{red}{#1}}
\nc{\bluetext}[1]{\textcolor{blue}{#1}}
\nc{\brown}[1]{\browntext{ #1}}
\nc{\green}[1]{\greentext{ #1}}
\nc{\red}[1]{\redtext{ #1}}
\nc{\blue}[1]{\bluetext{ #1}}

\def \TT{\mathbf T}

\def \bF{\mathbf{F}}

\def \bvs{{\boldsymbol{\varsigma}}}

\def \vs{\varsigma}
\def \U{\mathrm U}
\def \Ui{\mathrm{U}^\imath}

\def \bB{\mathbf{B}}
\def \bK{\mathbf{K}}
\def \bE{\mathbf{E}}
\def \bF{\mathbf{F}}
\def \bU{\mathbf U}
\def \bUi{\mathbf{U}^\imath}

\newcommand{\arxiv}[1]{\href{http://arxiv.org/abs/#1}{\tt arXiv:\nolinkurl{#1}}}


\begin{document}

\title[Braid group symmetries on Poisson algebras]{Braid group symmetries on Poisson algebras arising from quantum symmetric pairs}

\author[Jinfeng Song]{Jinfeng Song}
\address{Department of Mathematics, National University of Singapore, Singapore}
\email{j$\_$song@u.nus.edu}

\author[Weinan Zhang]{Weinan Zhang}
\address{Department of Mathematics and New Cornerstone Science Laboratory, The University of Hong Kong, Pokfulam, Hong Kong SAR, P.R.China}
\email{mathzwn@hku.hk}
\maketitle

\begin{abstract}
Let $(\U,\Ui)$ be the quantum symmetric pair of arbitrary finite type and $G^*$ be the associated dual Poisson-Lie group. Generalizing the work of De Concini and Procesi \cite{DCP93}, the first author introduced an integral form for the $\imath$quantum group $\Ui$ and its semi-classical limit was shown to be the coordinate algebra for a Poisson homogeneous space of $G^*$. In this paper, we establish (relative) braid group symmetries and PBW bases on this integral form of $\Ui$. By taking the semi-classical limit, we obtain braid group symmetries and  polynomial generators on the associated Poisson algebra. These symmetries further allow us to describe the Poisson brackets explicitly. Examples of such Poisson structures include Dubrovin-Ugaglia Poisson brackets.
\end{abstract}

 \setcounter{tocdepth}{1}
\tableofcontents

\section{Introduction}

\subsection{Background}
For a complex semisimple Lie algebra $\g$, the associated Drinfeld-Jimbo quantum group $\U=U_q(\g)$ is a central object in modern mathematics. Let $\A=\mathbb{Z}[q^{\pm1/2}]$. The quantum group $\U$ admits two well-known integral $\A$-forms: one, introduced by Lusztig \cite{Lus90a}, specializes to the universal enveloping algebra $U(\g)$ at $q=1$; the other arises from the work of De Concini--Kac \cite{DK90} in the study of the quantum groups at roots of unity. It was proved by De Concini--Procesi \cite[Theorem~12.1]{DCP93} that after rescaling, the De Concini--Kac form specializes at $q=1$ to the coordinate algebra $\C[G^*]$ for the dual Poisson-Lie group $G^*$. This builds a direct connection between the rescaled De Concini--Kac form and Poisson geometry; see also \cite{Qin20,Sh22a,Sh22b} for its applications in the cluster algebras. We refer to the rescaled De Concini--Kac form as the DCKP-integral form, and denote it by $\U_\A$. 


Let $\theta$ be an algebra involution on $\g$ and $\g^\theta\subset\g$ be the fixed-point subalgebra. Associated with a symmetric pair $(\mathfrak{g},\mathfrak{g}^\theta)$, the \emph{quantum symmetric pair} $(\U,\U^\imath)$, introduced by Letzter \cite{Let02}, consists of the quantum group $\U$ and a coideal subalgebra $\Ui\subset \U$, called an \emph{$\imath$quantum group}. Proposed by Bao--Wang \cite{BW18a}, $\imath$quantum groups can be viewed as vast generalizations of quantum groups, and the $\imath$program aims at generalizing fundamental constructions on quantum groups to $\imath$quantum groups; see \cite{Wa23} for some progress in the $\imath$program.

The Lusztig-type integral form on modified $\imath$quantum groups was introduced by Bao and Wang \cite{BW18b} in the study of the $\imath$canonical basis, and it has been extensively studied over the last decade in representation theory, categorification, and geometry; see, for example, \cite{BW18a, BS22, Wa23}.

On the other hand, to relate $\imath$quantum groups with Poisson geometry, one is forced to consider the DCKP-type integral form $\Ui_\A$, which was introduced recently by the first author \cite{So24}, defined by $\Ui_\A=\U_\A \cap \Ui$. This integral form $\Ui_\A$ specializes to the coordinate algebra of the Poisson homogeneous space $K^\perp \backslash G^*$ \cite[Theorem 1]{So24}, and it has revealed exciting connections with cluster algebras \cite{So23}. Unlike the Lusztig-type integral form, the DCKP-type integral form $\Ui_\A$ has not been deeply studied, and its basic algebraic properties remain unclear. 

\subsection{Goal} In the current paper, we study the DCKP-type integral form for $\Ui$, as well as its semi-classical limit. We establish relative braid group symmetries and PBW bases on this integral form. 
By taking the semi-classical limit, we obtain relative braid group symmetries and a set of polynomial generators on the coordinate Poisson algebra $\C[K^\perp\backslash G^*]$. This allows us to describe the Poisson bracket on $K^\perp\backslash G^*$ explicitly.

\subsection{Braid group actions on integral forms of $\imath$quantum groups}

It was shown in \cite{DCP93} that $\U_\A$ is invariant under Lusztig's braid group symmetries and that there exists a rescaled PBW basis of $\U_\A$ established using braid group symmetries. As a generalization of Lusztig braid group symmetries on quantum groups, relative braid group symmetries $\TT_i$ on $\Ui$ were systematically constructed for arbitrary finite type by Wang and the second author \cite{WZ23}; see also \cite{KP11,Do20} for some previous case-by-case constructions via brute-force computation for most quasi-split type. Using the relative braid group symmetries, a PBW basis was constructed on $\Ui$ in \cite{LYZ24}.
 
Therefore, it is natural to expect that the DCKP-type integral form $\Ui_\A$ is invariant under the relative braid group symmetries and admits a (rescaled) PBW basis in general. The first main result of this paper settles this problem for $\Ui_{\A'}$ where $\A'$ is a certain localization of $\A$  (see Section \ref{sec:proj}) and $\Ui_{\A'}=\Ui_\A \otimes_\A \A'$.

\begin{thmintro} \label{thm:1}(Theorem~\ref{thm:intt} \& Theorem~\ref{thm:PBW})
    Let $(\U,\Ui)$ be a quantum symmetric pair of arbitrary finite type. 
    \begin{itemize}
    \item[(1)] The relative braid group symmetries $\TT_i$ on $\Ui$ preserve the integral form $\Ui_{\A'}$. 
    \item[(2)] There exists a rescaled PBW basis for $\Ui_{\A'}$. In particular, the root vectors arising from the rescaled PBW basis form a finite generating set of $\Ui_{\A'}$ as an $\A'$-algebra.
    \end{itemize}
\end{thmintro}

Let us briefly explain the ideas in the proof. The difficulty is that $\Ui_{\A'}$ is not automatically equipped with a finite generating set and it is hard to directly establish a generating set without the relative braid group symmetries. However, in the $q=1$ case, if one allows the use of Poisson brackets, then a finite set of Poisson generators for $\C[K^\perp\backslash G^*]$ has been obtained in \cite{So23}. Motivated by this result, we introduce the {\em rescaled $q$-commutator} as a $q$-analog of the Poisson bracket; see Section~\ref{sec:qcom}. Then we show in Proposition~\ref{prop:gci} that there is a finite set $\mathcal{G}$ such that any element of $\Ui_{\A'}$ can be obtained by taking algebraic operations and rescaled $q$-commutators of elements in $\mathcal{G}$. Therefore, to show the integrability of $\TT_i$, it suffices to check formulas of $\TT_i(x)$ for $x\in \mathcal{G}$ and this can be done directly.

In our approach, it is necessary to invert $1+q$ and quantum Cartan integers, and this leads to the localization $\A'$. For further applications on Poisson geometry and quantum symmetric pairs at roots of unity, this localization does not affect much.

\subsection{Braid group actions on the Poisson algebra}
Let $\P$ be the integral Poisson algebra obtained by the specialization $q^{1/2}\mapsto1$ of $\Ui_{\A'}$. The root vectors of $\U_{\A'}$ in Theorem~\ref{thm:1} descend to a set of polynomial generators for $\P$ over a Laurent polynomial ring, and the relative braid group symmetries $\TT_i$ on $\Ui_{\A'}$ descend to Poisson algebra automorphisms on $\P$; see Propositions \ref{prop:poly}-\ref{cor:rbg}.

Let $\P_\C=\C\otimes\P$. By \cite[Theorem~1]{So24}, the Poisson variety $\text{Spec\;}\P_\C$ is canonically isomorphic to the Poisson homogeneous space $K^\perp\backslash G^*$ associated to the underlying symmetric pair for $(\U,\Ui)$; cf. Section \ref{sec:psk}. Denote by $W^\circ\subset W$ the associated relative Weyl group with simple reflections $\bs_i$; see Section~\ref{sec:rbga}. Set $\sigma_w$ for $w\in W$ to be the braid group symmetries on $G^*$ established in \cite{DCKP92}; see Section~\ref{sec:plg}.


\begin{thmintro}\label{thm:2} (Theorem~\ref{thm:rbs})
    Relative braid group symmetries $\TT_i$ on $\Ui_{\A'}$ induce Poisson automorphisms $\sigma_i^{\imath}$ on $K^\perp\backslash G^*$ which satisfy braid relations in the relative Weyl group $W^\circ$.
\end{thmintro}



Our construction of the Poisson automorphism $\sigma_i^\imath$ is purely algebraic, and it is natural to expect a geometric description of these symmetries. The next main result of the paper settles this problem for quasi-split cases.

\begin{thmintro}\label{thm:3}(Proposition~\ref{prop:spg}, Theorem~\ref{thm:iw})
    Assume that $\I_\bullet=\emptyset$. Then there exist a Poisson anti-group involution $\Theta:G^*\rightarrow G^*$ and an embedding $\psi:K^\perp\backslash G^*\rightarrow G^*$, such that $\psi$ is an isomorphism onto the identity component of the $\Theta$-fixed locus of $G^*$. For any $i\in\I$, the map $\sigma_i^\imath$ is the unique Poisson automorphism on $K^\perp\backslash G^*$ such that the diagram  
    \begin{equation}
        \begin{tikzcd}
            & K^\perp\backslash G^* \arrow[r,"\psi"] \arrow[d,"\sigma_i^\imath"'] & G^* \arrow[d,"\sigma_{\bs_i}"] \\ & K^\perp\backslash G^* \arrow[r,"\psi"] & G^*   
        \end{tikzcd}
    \end{equation}
    commutes.
\end{thmintro}

Let us remark that the pair $(G^*,\Theta)$ is a \emph{symmetric Poisson group} in the sense of Xu \cite{Xu03}, and hence $\psi$ identifies $K^\perp\backslash G^*$ with a \emph{Dirac submanifold} of $G^*$ (cf.\cite{Xu03}).

We describe in Section~\ref{sec:braidex} the geometric relative braid group symmetries on $K^\perp\backslash G^*$ explicitly for quasi-split types, and computed in Section~\ref{sec:exa} the explicit Poisson brackets on $\P$ in various examples in terms of the polynomial generators arising from the rescaled PBW basis of $\U_{\A'}$.

The Poisson algebra $\P$ and its geometric model $K^\perp\backslash G^*$ provide a large family of algebraic Poisson structures, some of which have appeared sporadically in the literature. When the symmetric pair is $(\mathfrak{sl}_n,\mathfrak{so}_n)$, Boalch \cite{Boa01} showed that the associated Poisson brackets coincide with the Dubrovin--Ugaglia Poisson brackets, which arises from the study of Frobenius manifolds \cite{Du94,Uga99}.

In their study of the R-matrix realization for quantum symmetric pairs, Molev and Ragoucy \cite{MR08} constructed braid group symmetries on the Poisson algebras associated with $(\mathfrak{sl}_n,\mathfrak{so}_n)$, and investigated the Poisson algebras associated with $(\mathfrak{sl}_{2n}, \mathfrak{sp}_{2n})$. More recently, Zhang-Lin-Zhang \cite{ZLZ25} studied the Poisson structure and braid group symmetries associated with $(\mathfrak{sp}_{2n}, \mathfrak{gl}_n)$, generalizing Molev-Ragoucy's approach. 

Our work provides a unified and intrinsic construction for (relative) braid group symmetries symmetries on the Poisson algebras associated to arbitrary symmetric pairs. These symmetries and polynomial generators for the Poisson algebras will serve as tools for future study of the Poisson structure on $K^\perp\backslash G^*$.

When the symmetric pair is of diagonal type, our Theorem \ref{thm:2} recovers braid group symmetries constructed by De Concini--Kac--Procesi \cite{DCKP92, DCP93} on $G^*$; see Example~\ref{ex:diagonal} and Section~\ref{sec:braidex}(ii). 

\subsection{Future work} 

The DCKP-type integral forms are closely related to the quantum groups at roots of unity. In \cite[Chapter~5]{DCP93}, the authors showed that the Poisson algebra $\C[G^*]$ can be identified with a central subalgebra of $\U_\epsilon$, the quantum group at roots of unity. This central subalgebra is crucial in the study of representations of $\U_\epsilon$ (cf. \cite{DK90}). In a sequel, we will study the relation between the Poisson algebra $\P$ and the $\imath$quantum group at roots of unity.

The rescaled PBW bases are crucial for the cluster realization of quantum groups (see, for example, \cite{Sh22a}) and $\imath$quantum groups of type $\textrm{AI}$ \cite{So23}. We expect that our results provide fundamental ingredients for the cluster realization of $\imath$quantum groups of general types.

In the quantum group case, Boalch \cite{Boa02} obtained a description of the braid group symmetries on $G^*$ from the perspective of connections on Riemann surfaces. It would be interesting to see an analogous description for  $K^\perp\backslash G^*$ in the case of quantum symmetric pairs.

\subsection{Organization} The paper is organized as follows. In Section \ref{sec:pre}, we collect known results on quantum groups, quantum symmetric pairs, and Poisson-Lie groups which are needed in the paper. In Section \ref{sec:intb}, we show that relative braid group symmetries preserve the DCKP-type integral form $\Ui_{\A'}$ and establish PBW bases on this integral form. In Section \ref{sec:ppa}, we study the induced relative braid group symmetries on the Poisson homogeneous space $K^\perp\backslash G^*$. In Section \ref{sec:exa}, we provide in various examples the explicit Poisson brackets on $K^\perp\backslash G^*$ in terms of polynomial generators arising from PBW bases of $\Ui_{\A'}$. In Appendix \ref{app:A}, we provide explicit formulas for real rank 1 root vectors and verify that they belong to the integral form. 

\vspace{.2cm}
\noindent {\bf Acknowledgement: } The authors thank Huanchen Bao for helpful discussions. WZ thanks Ming Lu for discussions on the PBW basis of $\imath$quantum groups. JS is partially supported by Huanchen Bao’s MOE grant A0004586-00-00 and A-0004586-01-00. WZ is partially supported by the New Cornerstone Foundation through the New Cornerstone Investigator grant awarded to Xuhua He.

\section{Preliminaries}\label{sec:pre}
\subsection{Quantum groups}\label{sec:qgp}


Let $\g$ be a semisimple Lie algebra associated to the Cartan matrix $(a_{ij})_{i,j\in \I}$. Let $W$ be the Weyl group with simple reflections $s_i,i\in \I$ and $\cR$ be the root system for $\g$. Fix a set of simple roots $\{\alpha_i|i\in \I\}$ and set $\cR^+$ to be the corresponding set of positive roots. Let $\cR^-=-\cR^+$. For $w\in W$, we set $\cR^+(w)=\cR^+\cap w(\cR^-)$. Let $\{h_i\mid i\in\I\}$ be the set of simple coroots. 
Let $\{\varepsilon_i|i\in \I\}$ be the set of coprime positive integers such that $\varepsilon_ia_{ij}=\varepsilon_j a_{ji}$ for all $i,j\in \I$. 

Let $q$ be an indeterminate and $\Q(q)$ be the field of rational functions in $q$ with coefficients in $\C$, the field of complex numbers. Fix a square root $q^{1/2}$ of $q$ in the algebraically closed field of $\Q(q)$. 
We denote
\[
q_i:=q^{\epsilon_i}, \qquad \forall i\in \I.
\]
Denote, for $r,m \in \N,i\in \I$,
\[
 [r]_i =\frac{q_i^r-q_i^{-r}}{q_i-q_i^{-1}},
 \quad
 [r]_i!=\prod_{s=1}^r [s]_i, \quad \qbinom{m}{r}_i =\frac{[m]_i [m-1]_i \ldots [m-r+1]_i}{[r]_i!}.
\]

The Drinfeld-Jimbo quantum group $\U$ is defined to be the $\Q(q^{1/2})$-algebra generated by $E_i,F_i, K_i^{\pm 1}$ for $i\in\I$, subject to the following relations:
\begin{align}
&K_i K_i^{-1}=K_i^{-1} K_i=1,\; K_{i}K_{j}=K_{j} K_i, \label{eq:KK}
\\
&K_{j}E_i=q_j^{a_{ji}} E_i K_j, \quad K_j F_i=q_j^{-a_{ji}}F_i K_j,\\
&[E_i,F_j]=\delta_{ij}(q_i-q_i^{-1})(\tK_i^{-1}-\tK_i), \label{Q4}
\end{align}
and the quantum Serre relations, for $i\neq j \in \I$,
\begin{align}
& \sum_{r=0}^{1-a_{ij}} (-1)^r  \qbinom{1-a_{ij}}{r}_i E_i^{r} E_j  E_i^{1-a_{ij}-r}=0,
  \label{eq:serre1} \\
& \sum_{r=0}^{1-a_{ij}} (-1)^r  \qbinom{1-a_{ij}}{r}_i F_i^{r} F_j  F_i^{1-a_{ij}-r}=0.
  \label{eq:serre2}
\end{align}

\begin{remark}\label{rmk:bU}
    Our convention of generators is similar to \cite{Sh22b,So23} and differs from the standard ones. Let $\mathbf{U}$ be the quantum group in \cite{Lus93} with generators $\mathbf{F}_i$, $\mathbf{E}_i$, $\mathbf{K}_i^{\pm1}$, for $i\in \I$. Then one has an algebra isomorphism $\U\rightarrow\mathbf{U}$, given by 
    \[
    F_i\mapsto q_i^{1/2}(q_i-q_i^{-1})\bF_i,\quad E_i\mapsto q_i^{-1/2}(q_i^{-1}-q_i)\bE_i,\quad K_i\mapsto \mathbf{K}_i,\quad (i\in\I).
    \]
\end{remark}

Let $\U^+, \U^0, \U^-$ be the subalgebras of $\U$ generated by $\{E_i\mid i\in \I\}, \{K_i, K_i^{-1}\mid i\in \I\}, \{F_i\mid i\in \I\}$, respectively. There is a natural triangular decomposition
\begin{equation}\label{eq:tri}
\U\cong \U^+\otimes \U^0\otimes \U^-
\end{equation}
 We also have a natural $\Z \I$-grading on $\U$ given by
\begin{align}\label{eq:grad}
\deg F_i = \alpha_i,\quad \deg E_i = -\alpha_i,\quad \deg K_i = 0,\qquad (i\in \I).
\end{align}

For $i\in\I$, let $T_i$ be Lusztig's braid group symmetries on $\U$ \cite{Lus93}, which we recall as follows.
\begin{align*}
    &T_i(E_i)= q_i^{-1}F_iK_i,\qquad T_i(F_i)= q_iK_i^{-1}E_i;
    \\
    &T_i(E_j)=\frac{1}{q_i^{a_{ij}/2}(q_i^{-1}-q_i)^{-a_{ij}}[-a_{ij}]_i!}\sum_{r=0}^{-a_{ij}}(-1)^rq_i^{-r} \qbinom{-a_{ij}}{r}_i E_i^{-a_{ij}-r}E_jE_i^{r}\quad \text{for }j\neq i;\\
    & T_i(F_j)=\frac{1}{q_i^{-a_{ij}/2}(q_i-q_i^{-1})^{-a_{ij}}[-a_{ij}]_i!}\sum_{r=0}^{-a_{ij}}(-1)^rq_i^{r} \qbinom{-a_{ij}}{r}_i F_i^{r}F_jF_i^{-a_{ij}-r}\quad \text{for }j\neq i;\\
    & T_i(K_j)=K_{j}K_i^{-a_{ij}}.
\end{align*}

\begin{remark}
$T_i$ corresponds to the symmetry $T_{i,+1}''$ in \cite{Lus93}. The convention of braid group symmetry is different from \cite{So23}, but compatible with \cite{Jan96}.
\end{remark}
We define $T_w:= T_{i_1}\cdots T_{i_r},$ for a reduced expression $w = s_{i_1}\cdots s_{i_r}$ of $w \in W$.
Fix a reduced expression ${w_0}=s_{i_1} \cdots s_{i_l}$ for the longest element $w_0\in W$. For $1\le k\le l$, we set
\begin{align}
\label{eq:F-root}
&F_{\beta_k}=T_{i_1}T_{i_2}\cdots T_{i_{k-1}}(F_{i_k}),\qquad E_{\beta_k}=T_{i_1}T_{i_2}\cdots T_{i_{k-1}}(E_{i_k}),
\end{align}
to be the root vectors associated with the root $\beta_k=s_{i_1}s_{i_2}\cdots s_{i_{k-1}}(\alpha_{i_k})\in \cR^+$.
For $\ba=(a_1,\ldots, a_l)\in \N^l$, we set
\begin{equation}
  F^{\ba}=F_{\beta_1}^{a_1} \cdot F_{\beta_2}^{a_2}\cdots F_{\beta_l}^{a_l},
  \qquad 
  E^{\ba}=E_{\beta_l}^{a_1} \cdot E_{\beta_{l-1}}^{a_2}\cdots E_{\beta_1}^{a_1}.
\end{equation}

Set $\A=\Z[q^{1/2},q^{-1/2}]$. Let $\U_{\A}^- $ (resp., $\U_{\A}^+$) be the $\A$-module of $\U^-$ (resp., $\U^+$) spanned by $F^{\ba},a\in \N^\ell$ (resp., $E^{\ba},a\in \N^\ell$). It is known that $\U_{\A}^-, \U_{\A}^+ $ are independent of the choice of the reduced expression for $w_0$ and hence are $\A$-subalgebras. Let $\U_{\A}^0$ be the $\A$-subalgebra of $\U$ generated by $K_i,i \in \I$. Set $\U_\A $ to be the $\A$-subalgebra generated by $\U_{\A}^-,\U_{\A}^0,\U_{\A}^+$. For $\mu=\sum_{i\in \I } b_i \alpha_i\in\Z\I$, set $K_\mu=\prod_{i\in \I} K_i^{b_i}$.

\begin{proposition}[\cite{BG17}]
\label{prop:dec}
There is an isomorphism of $\A$-modules given by multiplication
\begin{align*}
\U_{\A}\cong\U_{\A}^-\otimes \U_{\A}^0 \otimes \U_{\A}^+.
\end{align*}
The set of monomials
\begin{align}
\{F^{\ba} K_{\mu} E^{\bc}|\mu\in \Z\I,\ba,\bc \in \N^l \}
\end{align}
forms an $\A$-basis of $\U_\A$, which is called the (rescaled) \emph{PBW} basis.
\end{proposition}


\subsection{Quantum symmetric pairs}\label{sec:iqg}

A Satake diagram $(\I=\bI \cup \wI,\tau)$ (cf. \cite{Ko14}) consists of a partition $\bI\cup \wI$ of $\I$, and an involution $\tau$ of $\I$ (possibly $\tau=\Id$) such that
\begin{itemize}
\item[(1)] $a_{ij}=a_{\tau i,\tau j}$ for $i,j\in\I$, 
\item[(2)] $\tau(\bI)=\bI$,
\item[(3)]
 $\bw(\alpha_j) = - \alpha_{\tau j}$ for $j\in \bI$,
\item[(4)]
 If $j\in \wI$ and $\tau j =j$, then $\langle \rho_{\bullet}^\vee,\alpha_j\rangle\in \Z$.
\end{itemize}
Here $\bw$ is the longest element in the Weyl group $W_\bullet$ associated to $\bI$ and $\rho_{\bullet}^\vee$ is the half sum of positive coroots in the subcoroot system generated by $\I_\bullet$. 

A Satake diagram  $(\I=\bI \cup \wI,\tau)$ is called quasi-split if $\bI=\emptyset$. In this case, we denote it by $(\I,\tau)$.

A symmetric pair $(\g,\g^\theta)$ (of finite type) consists of a semisimple Lie algebra $\g$ and the fixed point subalgebra $\g^\theta$ for an involution $\theta$ on $\g$. For each Satake diagram, there is an associated involution $\theta$ on $\g$, whose action on the weight lattice is given by $\theta=-\bw \circ \tau$. The irreducible symmetric pairs are classified by Satake diagrams.


Given a Satake diagram $(\I=\wI\cup\bI,\tau )$, the $\imath$quantum group $\Ui$ is the $\Q(q^{1/2})$-subalgebra of $\U$ generated by
\begin{align}
\label{def:iQG}
\begin{split}
&B_i=F_i-c_i q_i^{-\langle h_i,\bw \alpha_{\tau i} \rangle/2} T_{\bw}( E_{\tau i}) K_i^{-1}, \qquad k_i = K_i K_{\tau i}^{-1} \qquad (i\in \wI),
\\
&E_j,\quad B_j=F_j,\quad K_j^{\pm 1} \qquad (j\in \bI).
\end{split}
\end{align}

Here $c_i\in\{\pm 1\}$ ($i\in\I_\circ$) such that 
\begin{align}\label{def:ci}
\begin{split}
 c_i c_{\tau i}&=(-1)^{\langle 2\rho_\bullet^\vee,\alpha_i\rangle},\quad \text{ for } i\in \I_\circ;
 \\
c_i&=c_{\tau i}, \quad \text{ if } \langle h_i,\theta(\alpha_i)\rangle=0.
\end{split}
\end{align}

\begin{remark}[\text{cf. \cite[Remark 2.5]{So23}}]\label{rmk:bUi}
The $\imath$quantum group $\bUi_\bvs$ introduced in \cite{Let02,Ko14} with parameters $\bvs=(\vs_i)_{i\in \wI}\in \big(\C(q^{1/2})^{\times}\big)^{\wI}$ is the subalgebra of $\bU$ generated by
\[
\bB_i=\bF_i+ \vs_i T_{\bw}( \bE_{\tau i}) \mathbf{K}_i^{-1}, \quad 
\mathbf{k}_i=\bK_i \bK_{\tau i}^{-1}, \quad \bE_j, 
\quad \bF_j, \quad\mathbf{K}_j^{\pm 1},
\]
for $i\in \wI,j\in \bI$. 
Set $\vs_i=c_i q_i^{-\langle h_i,w_\bullet\alpha_{\tau i}\rangle/2-1}$ for $i\in \wI$, and then the isomorphism in Remark~\ref{rmk:bU} induces an algebra isomorphism
\begin{align*}
\begin{split}
\Omega:&\Ui\rightarrow \bUi_\bvs\\
&B_i\mapsto q_i^{1/2}(q_i-q_i^{-1})\bB_i,\quad k_i\mapsto \mathbf{k}_i,\qquad (i\in \wI),\\
&E_j\mapsto q_j^{-1/2}(q_j^{-1}-q_j)\bE_j,\quad  F_j\mapsto q_j^{1/2}(q_j-q_j^{-1})\bF_j, \quad K_j \mapsto \mathbf{K}_j,\qquad (j\in \bI).
\end{split}
\end{align*}
In this paper, for the sake of integrality, we shall mainly consider $\Ui$, the $\imath$quantum group with these specific parameters.
Note that if $c_i=-1$, then $\vs_i$ is the distinguished parameter used in \cite{WZ23}.
\end{remark}

\begin{remark}
The value of $\langle 2\rho_\bullet^\vee,\alpha_i\rangle$ appearing in \eqref{def:ci} was calculated in \cite[\S 3.3]{BW18b}. When the rank 1 Satake subdiagram associated to $i$ is not of type AIV, $\langle 2\rho_\bullet^\vee,\alpha_i\rangle$ is always an even integer; hence, in these cases, one can freely choose $c_i$ to be $1$ or $-1$. Note also that the rank 1 Satake subdiagram associated to $i$ not being of type AIV is equivalent to saying that $\tau i = i$ or $\bw (\alpha_i)=\alpha_i$; see Table~\ref{table:localSatake}.
\end{remark}

For convenience, unless otherwise specified, we shall always assume that 
\begin{align} 
c_i=-1, \text{ if } \tau i = i \text{ or } \bw (\alpha_i)=\alpha_i.
\end{align}


We set $\U^{\imath 0}\subset \Ui$ to be the unital subalgebra generated by $k_i^{\pm1}$ for $i\in\wI$, and $K_j^{\pm1}$ for $j\in\I_\bullet$. Following \cite{So23,So24}, we define 
\begin{align}\label{def:UiA}
\Ui_{\A}:=\Ui\cap \U_{\A}.
\end{align}
As shown {\em loc. cit.}, $\Ui_{\A}$ can be view as a generalization of the De Cocini-Kac-Procesi integral form $\U_\A$ in the theory of quantum symmetric pairs.
We therefore call $\Ui_{\A}$ the {\em DCKP-type integral form} for the $\imath$quantum group $\Ui$.

\begin{proposition}\label{prop:Tj}
The algebra automorphisms $T_j$, $T_j^{-1}$ for $j\in \bI$ restricts to algebra automorphisms on $\Ui$. Moreover, $T_j$, $T_j^{-1}$ for $j\in \bI$ preserves $\Ui_{\A}$.
\end{proposition}

\begin{proof}
The first statement follows from \cite[Theorem 4.2]{BW18b}. The second statement follows from the first statement and $T_j(\U_{\A}) \subset \U_{\A}$.
\end{proof}

\subsection{A projection map}\label{sec:proj}

For a quantum symmetric pair $(\U,\Ui)$, we shall consider a localization $\A'$ of $\A=\mathbb{Z}[q^{1/2},q^{-1/2}]$ defined as follows. If  $(\I=\wI\cup\bI,\tau)$ contains a subdiagram of type AIV$_n$, set 
$$
\A'=\A[\sqrt{(-1)^n}, (1+q)^{-1}, [\varepsilon_i]!^{-1}\mid i\in\I],
\qquad \Z'=\Z[\sqrt{(-1)^n}, 2^{-1},\varepsilon_i^{-1}\mid i\in\I]
$$
to be the subrings of $\Q(q^{1/2})$ and $\C$, respectively.
 If  $(\I=\wI\cup\bI,\tau)$ does not contain a subdiagram of type AIV, set
 $$
\A'=\A[(1+q)^{-1}, [\varepsilon_i]!^{-1}\mid i\in\I],
\qquad
\Z'=\Z[2^{-1},\varepsilon_i^{-1}\mid i\in\I]
$$ 
to be the subrings of $\Q(q^{1/2})$ and $\C$, respectively. 
Then $\Z'$ can be viewed as an $\A'$-module via the ring homomorphism $\A'\rightarrow \Z', q\mapsto 1$. We define $\U_{\A'}, \U^\pm_{\A'}$, $\U_{\A'}^0$ and $\Ui_{\A'}$ via the evident base change. The rings $\A'$ and $\Z'$ will be crucial in the later part.

\begin{remark}
When the underlying Satake diagram is of quasi-split ADE type, the ring $\A'$ is simplified as $\A'=\A[(1+q)^{-1}]$. 
\end{remark}


Let $\U^+(w_\bullet)\subset\U^+$ be the unital $\Q(q^{1/2})$-subalgebra generated by $E_i$, for $i\in\I_\bullet$. Let $\U^+(w_\bullet)^c=\U^+\cap T_{w_\bullet}(\U^+)$. We set $\U_{\A'}^+(w_\bullet)=\U^+(w_\bullet)\cap{\U_{\A'}}$ and $\U_{\A'}^+(w_\bullet)^c=\U^+(w_\bullet)^c\cap\U_{\A'}$. Fix a reduced expression of $w_0$ starting with a reduced expression of $w_\bullet$. By considering the PBW basis of $\U^+$ in Proposition \ref{prop:dec}, one concludes that the multiplication gives a decomposition $\U^+\cong\U^+(w_\bullet)\otimes\U^+(w_\bullet)^c$ as vector spaces, and moreover one has the decomposition $\U_{\A'}^+\cong\U_{\A'}^+(w_\bullet)\otimes\U_{\A'}^+(w_\bullet)^c$ as $\A'$-modules. Set 
$$
\pi^+_{w_\bullet}=(id)\otimes\epsilon:\U^+\rightarrow\U^+(w_\bullet),
$$ 
where $\varepsilon:\U^+(w_\bullet)^c\rightarrow\Q(q^{1/2})$ is the restriction of the counit map. It is clear that $\pi_{w_\bullet}^+$ restricts to $\U^+_{\A'}\rightarrow\U_{\A'}^+(w_\bullet)$. 

Let $\wItau$ be a fixed set of representatives for $\tau$-orbits on $\wI$, and let $\U^{0,c}\subset\U^0$ be the unital subalgebra generated by $K_i^{\pm1}$ for $i\in\wItau$. Recall the subalgebra $\U^{\imath0}$ from Section \ref{sec:iqg}. It is clear that $\U^0\cong\U^{\imath0}\otimes\U^{0,c}$ and $\U_{\A'}^0\cong\U_{\A'}^{\imath0}\otimes\U_{\A'}^{0,c}$, where $\U_{\A'}^{\imath0}=\U^{\imath0}\cap\U_{\A'}$ and $\U_{\A'}^{0,c}=\U^{0,c}\cap\U_{\A'}$. Under this decomposition of $\U^0$, we let 
$$
\pi^0=(id)\otimes\epsilon:\U^0\rightarrow\U^{\imath0}.
$$ It is clear that $\pi^0$ restricts to the map $\U^0_{\A'}\rightarrow\U^{\imath0}_{\A'}$.

Let $\U_P$ be the unital subalgebra of $\U$ generated by elements $E_j,F_i$ for $j\in\I_\bullet,i\in\I$ and $\U^{\imath0}$. Then \eqref{eq:tri} restricts to a triangular decomposition of $\U_P$, that is,
\begin{equation}\label{eq:trip}
\U_P\cong \U^+(w_\bullet)\otimes\U^{\imath0}\otimes \U^-.
\end{equation}
Set $\U_{P,\A'}=\U_P\cap \U_{\A'}$. Thanks to Proposition \ref{prop:dec}, one has the following isomorphism of $\A'$-modules
\begin{equation}\label{eq:tripA}
\U_{P,\A'}\cong\U_{\A'}^+(w_\bullet)\otimes\U^{\imath0}_{\A'}\otimes\U^-_{\A'}.
\end{equation}

Under the triangular decomposition \eqref{eq:tri} and \eqref{eq:trip}, we define a surjecitve linear map
\begin{equation}\label{eq:pi}
    \pi_P=\pi^+_{w_\bullet}\otimes\pi^0\otimes(id):\U\longrightarrow \U_P.
\end{equation}
It follows from the above discussion that $\pi_P(\U_{\A'})\subset\U_{P,\A'}$, and that $\pi_P(x)=x$ for $x\in\U_P$. Let us define
\begin{equation}
    \pi^\imath=\pi_P\mid_{\U^\imath}:\Ui\longrightarrow \U_P,
\end{equation}
which restricts to the map
\begin{equation}\label{eq:pia}
    \pi^\imath_{\A'}:\Ui_{\A'}\longrightarrow\U_{P,\A'}.
\end{equation}

The map $\pi^\imath$ first appeared in \cite[Section 4.3]{Let19}. Following \cite[\S 2.5]{KY21}, we call $\pi^\imath$ the {\em Letzter map} and it is an isomorphism of $\Q(q^{1/2})$-vector spaces; see \cite[Lemma 2.10]{KY21}. We will show in Proposition \ref{prop:intp} that $\pi_{\A'}$ is an isomorphism as $\A'$-modules.

\subsection{Relative braid group action}\label{sec:rbga}

Recall that $\wItau\subset\wI$ is a fixed set of representatives for $\tau$-orbits. The real rank of a Satake diagram $(\I=\bI \cup \wI,\tau)$ is defined to be the cardinality of $\wItau$.

We recall the definition of relative Weyl group from \cite{Lus03,DK19}; notations here follows the convention in \cite{WZ23}. For $i\in \wI$, we set $\I_{\bullet,i}:=\bI\cup \{i,\tau i\}$ and there is a real rank 1 Satake subdiagram $(\I_{\bullet,i}=\{i,\tau i\}\cup\bI,\tau\big|_{\I_{\bullet,i}})$.
Let $\cR_{\bullet,i}$ be the set of roots which are linear combinations of $\alpha_j,j\in \I_{\bullet,i}$; $\cR_{\bullet,i}$ is naturally a root system with the simple system $\{\alpha_j\mid j\in \I_{\bullet,i}\}$. Denote by $\cR_{\bullet,i}^+$ the corresponding positive system of $\cR_{\bullet,i}$. Let $W_{\bullet,i}$ be the parabolic subgroup of $W$ generated by $s_i,i\in \I_{\bullet,i}$ and $w_{\bullet,i}$ the longest element of $W_{\bullet,i}$. Define $\tau_i$ such that $w_{\bullet,i}(\alpha_j)=-\alpha_{\tau_{i} j}$ for $j\in \I_{\bullet,i}$.

Define $\bs_i\in W_{\bullet,i}$ such that
\begin{align}\label{def:bsi}
\bwi= \bs_i w_\bullet \, (=w_\bullet \bs_i).
\end{align}
It is clear from the definition that $\bs_i=\bs_{\tau i}$. The relative Weyl group $W^\circ$ associated to $(\I=\bI \cup \wI,\tau)$ can be identified with the subgroup of $W$ generated by $\bs_i,i\in \wItau$. The group $W^\circ$ is itself a Coxeter group with Coxeter generators $\bs_i$. Let $\bbw_0$ be the longest element in $W^\circ$. It is known that $w_0=\bbw_0 w_\bullet$.

In \cite{WZ23}, the relative braid group symmetries are constructed on $\bUi$. Using the isomorphism $\Omega$ in Remark~\ref{rmk:bUi}, we obtain relative braid group symmetries $\TT_i$ on $\Ui$.

Recall from the definition of Satake diagram that $w_\bullet(\alpha_j)=-\alpha_{\tau j}$ for $j\in \bI$. It follows from the definition \eqref{def:bsi} that for $j\in \bI$, $\bs_i(\alpha_j)=\alpha_{\tau \tau_{i} j}$ is a simple root for $\cR_\bullet$.

\begin{proposition} {\em\cite[Lemma 9.2, Propositions 6.2-6.3, Theorems 6.1 and 9.9.]{WZ23}}
\label{prop:WZ} 
For $i\in\wItau$, there exist automorphisms $\TT_{i}$ on $\Ui$ which satisfy braid relations in $W^\circ$. Moreover, for $i\in \wItau,j\in \bI$, we have
\begin{itemize}
\item[(1)] $\TT_i T_j =T_{\tau \tau_{i} j} \TT_i$;
\item[(2)] $\TT_i(F_j)=F_{\tau \tau_{i} j}, \TT_i(E_j)=E_{\tau \tau_{i} j},\TT_{i}(K_j)=K_{\tau \tau_{i} j}$;
\item[(3)] $\TT_{i}(k_p)=K_{\bs_i(\alpha_p-\alpha_{\tau p})}$ for $p\in \wI$;
\item[(4)] $\TT_i(B_i)=c_i^{-1/2} c_{\tau i}^{1/2}q_i^{\langle h_i,\bw\alpha_{\tau i}\rangle-1} T_{\bw}^{-2}(B_{\tau \tau_i (i)})K_{\tau\tau_i(\alpha_i)-\bw\tau_i (\alpha_i)}$;
\item[(5)] $\TT_i(B_{\tau i})=c_i^{1/2} c_{\tau i}^{-1/2} q_i^{\langle h_{\tau i},\bw\alpha_{i}\rangle-1} T_{\bw}^{-2}(B_{\tau_i (i)})K_{ \tau_i(\alpha_i)-\bw\tau\tau_i (\alpha_i)}$.
\end{itemize}
\end{proposition}

Let $\sigma$ be the anti-involution on $\U$ such that $\sigma:E_i\mapsto E_i,F_i\mapsto F_i,K_i\mapsto K_i^{-1}$. By \cite[Proposition 3.14]{WZ23}, there exists an anti-involution $\sigma_\tau$ on $\Ui$ such that 
\begin{align}
\label{sigmatau}
\sigma_\tau(B_i)=B_{\tau i} \text{ and } \sigma_\tau(x)=\sigma\tau(x),
\forall x\in\U_{\bullet}\U^{\imath0}.
\end{align}

\begin{lemma}\label{lem:sigmaT}
We have $\TT_i^{-1}=\sigma_\tau \circ \TT_i \circ \sigma_{\tau}$ for any $i\in \wItau$.
\end{lemma}

\begin{proof}
Let $\widetilde{\U}^\imath$ be the universal $\imath$quantum group and recall the central reduction $ \mathrm{Red}: \widetilde{\U}^\imath \rightarrow \Ui$ at the distinguished parameters (i.e., $c_i=-1$). The map $ \mathrm{Red}$ was denoted by $\pi_{\bvs_\diamond}$ in \cite{WZ23}.
Let $\sigma^\imath$ be the anti-involution on $\widetilde{\U}^\imath$ defined in \cite[Proposition 3.2]{WZ23} and $\widetilde{\TT}_i$ be the relative braid group symmetries on $\widetilde{\U}^\imath$. By \cite[Remark 4.8]{WZ23}, we have $\tau\circ\TT_i\circ\tau=\TT_i$. By \cite[Theorem 6.7]{WZ23}, we have $\widetilde{\TT}_i^{-1}=\sigma^{\imath} \circ \widetilde{\TT}_i \circ \sigma^{\imath}$. Hence, 
\begin{align}\label{eq:uniQG}
\widetilde{\TT}_i^{-1}=\sigma^{\imath}\tau \circ \widetilde{\TT}_i \circ \sigma^{\imath}\tau
\end{align}
Note that $\sigma_\tau$ is obtained from $\sigma^\imath \circ \tau$ via the central reduction $ \mathrm{Red}$. 

If the underlying rank 1 Satake diagram associated to $i$ is not type AIV, then $c_i=c_{\tau i}=-1$, and by the construction in \cite[Section 9.4]{WZ23}, we have
\[
\TT_i\circ\mathrm{Red}=\mathrm{Red}\circ \widetilde{\TT}_i,
\qquad 
\TT_i^{-1}\circ\mathrm{Red}=\mathrm{Red}\circ \widetilde{\TT}_i^{-1}.
\]
Hence, applying $ \mathrm{Red}$ to \eqref{eq:uniQG}, we obtain the desired $\TT_i^{-1}=\sigma_\tau \circ \TT_i \circ \sigma_{\tau}$. 

If the underlying rank 1 Satake diagram associated to $i$ is type AIV, then $\TT_i\circ \mathrm{Red}$ differs from $\mathrm{Red}\circ \widetilde{\TT}_i$ by a further rescaling; cf. {\em loc. cit.}. In this case, one can use the formulas of $\TT_i$ and $\TT_i^{-1}$ in \cite{WZ23} to directly check that $\TT_i^{-1}=\sigma_\tau \circ \TT_i \circ \sigma_{\tau}$. This completes the proof.
\end{proof}

\begin{remark}
Let $\mathbf{U}^\imath_{\bvs}$ be the $\imath$quantum group with an arbitrary parameter $\bvs=(\vs_i)_{i\in \wI}$ \cite{Let02,Ko14}. One can use similar arguments as above to show that $\TT_i^{-1}=\sigma_\tau \circ \TT_i \circ \sigma_{\tau}$ on $\mathbf{U}^\imath_{\bvs}$.
\end{remark}

\subsection{Semi-classical limits and Poisson brackets}\label{sec:scl}

For a (not necessarily commutative) $\A'$-algebra $R_q$, suppose it satisfies the condition
\begin{equation}\label{eq:com}
    [f,g]=fg-gf\in(q^{1/2}-1) R_q,\qquad \forall f,g\in R_q.
\end{equation}
Let $R=\mathbb{Z}'\otimes _{\A'} R_q$ be the $\Z'$-algebra obtained by the base change. The condition \eqref{eq:com} implies that $R$ is commutative. Denote by $f\mapsto\overline{f}$ the canonical map $R_q\rightarrow R$. Then $R$ carries a Poisson bracket defined by
\begin{equation}\label{eq:poi}
    \{\overline{f},\overline{g}\}=\overline{\frac{fg-gf}{2(q^{1/2}-1)}},\qquad \forall f,g\in R_q.
\end{equation}
The Poisson algebra $R$ is called the \emph{semi-classical limit} of $R_q$.

\begin{proposition}[\text{\cite[\S 12.2]{DCP93}}]\label{prop:Ucom}
For $A,B$ in $\U_\A^+$ (resp., $\U^-_\A$, $\U_\A$), one has
\begin{equation}\label{eq:Ucom}
    [A,B]\in (q-1)\U_\A^+\quad \text{(resp., $\U^-_\A$, $\U_\A$)}.
\end{equation}
\end{proposition}

For $\mu\in\mathbb{N}\I$, let $\U_{\A,\mu}^-$ be the $\A$-submodule of homogeneous elements of degree $\mu$. By the PBW basis, $\U_{\A,\mu}^-$ is a free $\A$-module and $\U^-_\A=\oplus_{\mu\in\mathbb{N}\I}\U^-_{\A,\mu}$.
It is also known that $\U_\A$ is invariant under the action of $T_i,i\in \I$; cf. \cite[\S 12.1]{DCP93}.

\subsection{Poisson-Lie structures on $G^*$}\label{sec:plg}

A Poisson-Lie group is a complex group with a structure of a complex Poisson manifold, where the multiplication map is Poisson. We recall several Poisson structures related to complex semisimple groups in this section.

Let $\fg=\mathfrak{h}\oplus\bigoplus_{\alpha\in\mathcal{R}}\fg_\alpha$ be the root space decomposition of $\g$. Let  $\mathfrak{b}^+=\mathfrak{h}\oplus\bigoplus_{\alpha\in\mathcal{R}^+}\fg_\alpha$, and $\mathfrak{b}^-=\mathfrak{h}\oplus\bigoplus_{\alpha\in\mathcal{R}^-}\fg_\alpha$ be the Borel subalgebras. 

Let $G$ be the complex connected semisimple adjoint group associated with the Lie algebra $\g$, and $B^+$ (resp., $B^-$) be the Borel subgroup associated with Lie algebras $\mathfrak{b}^+$ (resp., $\mathfrak{b}^-$). Let $H=B^+\cap B^-$ be the maximal torus and $U^+$ (resp., $U^-$) be the unipotent radical of $B^+$ (resp., $B^-$). Let $D=G\times G$. Set 
\[
G_\Delta=\{(g,g)\mid g\in G\},\qquad G^*=\{(u_+h,h^{-1}u_-)\in B^+\times B^-\mid u_\pm\in U^\pm\}
\]
to be the subgroups of $D$.

Let $\mathfrak{d}=\fg\oplus\fg$ be the Lie algebra of $D$. Let $\fg^*$ and $\fg_\Delta$ be the Lie algebras of $G^*$ and $G_\Delta$, respectively, which are all identified with Lie subalgebras of $\mathfrak{d}$. 

Let $\k$ be the Killing form on $\fg$. Define the nondegenerate invariant symmetric bilinear form on $\mathfrak{d}$ by
\[
\langle (x_1,x_2),(y_1,y_2)\rangle=\k( x_1,y_1)-\k(x_2,y_2),\qquad x_1,x_2,y_1,y_2\in\fg.
\]
Then $\mathfrak{d}=\fg^*+\fg_\Delta$ is a Lagrangian splitting with respect to this bilinear form. Let us define
\[
\Pi_0=\sum_{i=1}^k\xi_i\wedge\xi^i\in\wedge^2\mathfrak{d},
\]
where $\{\xi_1,\dots,\xi_n\}$ and $\{\xi^1,\dots,\xi^n\}$ are pairs of dual bases of $\fg^*$ and $\fg_\Delta$ with respect to the bilinear form $\langle\;,\;\rangle$. We define bivector fields $\Pi$ on $D$ by
\[
\Pi(d)=\frac{1}{2}(\Pi^r_0-\Pi^l_0), \qquad d\in D.
\]
Here $\Pi_0^r$ and $\Pi_0^l$ are right and left invariant bivector field on $D$ with the value $\Pi_0$ at the identity. It is known that $(D,\Pi)$ is a Poisson-Lie group, and $G_\Delta$, $G^*$ are Poisson submanifolds of $(D,\Pi)$. Let $\Pi_{G^*}=\Pi\mid_{G^*}$. Then $(G^*,\Pi_{G^*})$ is a Poisson-Lie group, called the \emph{dual Poisson-Lie group} of $G$. The algebra $\mathbb{C}[G^*]$ of regular functions on $G^*$ naturally carries a Poisson bracket.

Consider the base change $\A'\rightarrow\mathbb{C}$ given by $q^{1/2}\mapsto 1$. By \cite[\S 12]{DCP93}, one has a canonical isomorphism
\begin{equation}\label{eq:icl}
\varphi: \C\otimes_{\A'}\U_{\A'} \overset{\sim}{\longrightarrow} \C[G^*]
\end{equation}
as Poisson algebras. 

We next recall the explicit construction of the isomorphism. For $\alpha\in\mathcal{R}$, we write $U_\alpha$ to be the root subgroup associated with the root $\alpha$. Fix a pinning $\{x_i:\mathbb{C}\rightarrow U_{\alpha_i},y_i:\mathbb{C}\rightarrow U_{-\alpha_i}\mid i\in\I\}$ of $G$. For $i\in \I$, let 
\[
\dot{s}_i=x_i(1)y_i(-1)x_i(1)\in G.
\]
For a reduced expression $w=s_{i_1}s_{i_2}\cdots s_{i_m}\in W$, let
\[
\dot{w}=\dot{s}_{i_1}\cdot \dot{s}_{i_2}\cdots\dot{s}_{i_m}\in G.
\]
It is known that the element $\dot{w}$ is independent of the choice of the reduced expression. For a reduced expression $w_0=s_{i_1}s_{i_2}\cdots s_{i_n}$ of the longest element in $W$, we set $w_{\le k}=s_{i_1}s_{i_2}\cdots s_{i_k}$ to be a reduced sub-expression, for $1\le k\le n$, and $w_{\le 0}=e$. Let $\beta_k=w_{\le k-1}\alpha_{i_k}$, and let $x_{\beta_ k}:\mathbb{C}\rightarrow U_{\beta_k}$, $y_{\beta_k}:\mathbb{C}\rightarrow U_{-\beta_k}$ be the maps given by $x_{\beta_k}(\xi)= \dot{w}_{\le k-1}\cdot x_{i_k}(\xi)\cdot \dot{w}_{\le k-1}^{-1}$, $y_{\beta_k}(\xi)= \dot{w}_{\le k-1}\cdot y_{i_k}(\xi)\cdot \dot{w}_{\le k-1}^{-1}$, for $1\le k\le n$ and $\xi\in\mathbb{C}$.

By \cite[Lemma 8.3.5]{Spr09}, elements $a\in U^+$, $b\in U^-$ can be factored as
\begin{equation}\label{eq:fac}
    a=x_{\beta_n}(a_n)x_{\beta_{n-1}}(a_{n-1})\cdots x_{\beta_{1}}(a_1)\quad \text{and}\quad b=y_{\beta_{1}}(b_1)y_{\beta_{2}}(b_{2})\cdots y_{\beta_{n}}(b_n),
\end{equation}
for unique tuples $(a_1,\cdots,a_n)\in\mathbb{C}^n$ and $(b_1,\cdots,b_n)\in\mathbb{C}^n$. We define $\chi_{\beta_k}^+\in\mathbb{C}[U^+]$ by $\chi_{\beta_k}^+(a)=a_k$, and define $\chi_{\beta_k}^-\in\mathbb{C}[U^-]$ by $\chi^-_{\beta_k}(b)=b_k$. If $\beta_{k}=\alpha_j$ is a simple root, elements $\chi^+_{\beta_k}$ and $\chi_{\beta_k}^-$ are independent of the reduced expression, in which case we denote by $\chi^+_j=\chi^+_{\beta_k}$ and $\chi_j^-=\chi_{\beta_k}^-$. Since $G$ has the trivial center, one has the canonical isomorphism $\mathbb{C}[H]\cong \mathbb{C}[\alpha_i^{\pm1}\mid i\in\I]$.

Recall the root vectors $F_{\beta_k}$, $E_{\beta_k}$ associated with the reduced expression $w_0=s_{i_1}s_{i_2}\cdots s_{i_n}$ in \eqref{eq:F-root}. Then the isomorphism \eqref{eq:icl} is given by 
\begin{align}
\begin{split}\label{eq:isg}
&\overline{E_{\beta_k}}:(u_+h,h^{-1}u_-)\mapsto \chi^+_{\beta_k}(u_+),\quad \overline{F_{\beta_k}}:(u_+h,h^{-1}u_-)\mapsto \chi_{\beta_k}^-(u_-),\\ &\overline{K_i}:(u_+h,h^{-1}u_-)\mapsto \alpha_i(h),
\end{split}
\end{align}
for $u_\pm\in U^\pm$, $h\in H$, $1\le k\le n$, and $i\in\I$.

The braid group actions $T_i$ ($i\in\I$) preserve the integral form $\U_\A$. They induce Poisson automorphism $\sigma_i$ on $G^*$ after specialization. The formula for $\sigma_i$ was obtained in \cite[\S 7.5]{DCKP92} which we now recall. 

For $i\in\I$, let $\varphi_i:\textrm{SL}_2(\mathbb{C})\rightarrow G$ be the group homomorphism such that 
\[
\varphi_i(\begin{pmatrix}
    1 & \xi \\ & 1
\end{pmatrix})=x_i(\xi),\qquad \varphi_i(\begin{pmatrix}
    1 & \\ \xi & 1
\end{pmatrix})=y_i(\xi),\qquad \text{for $\xi\in\mathbb{C}$}.
\]

Define the map $\sigma_i:G^*\rightarrow G^*$ by 
\begin{equation}\label{eq:gbg}
\sigma_i(u_+h,h^{-1}u_-)=(t_1u_+ht_2,t_1h^{-1}u_-t_2),
\end{equation}
where $t_1=\varphi_i(\begin{pmatrix}
    \alpha_i(h)\chi_i^-(u_-) & -1 \\ 1 &  
\end{pmatrix})$
and $t_2=\varphi_i(\begin{pmatrix}
    \alpha_i^{-1}(h)\chi_i^+(u_+) & 1 \\ -1 
\end{pmatrix})$.

The map $\sigma_i$ is invertible and its inverse $\sigma_i^{-1}:G^*\rightarrow G^*$ is given by 
\begin{equation}\label{eq:gbginv}
\sigma_i^{-1}(u_+h,h^{-1}u_-)=(t_1' u_+h t_2',t_1' h^{-1}u_- t_2'),
\end{equation}
where $t_1'=\varphi_i(\begin{pmatrix}
     & 1 \\ -1 &  \chi_i^+(u_+)
\end{pmatrix})$
and $t_2'=\varphi_i(\begin{pmatrix}
    & -1 \\ 1 &\chi_i^-(u_-) 
\end{pmatrix})$. 

\begin{remark}
 The map $\sigma_i^{-1}$ in our convention corresponds to the map $\sigma_i$ in \cite[Example 4.5]{Sh22b}.
\end{remark}
 
Then the comorphism $\sigma_i^*:\mathbb{C}[G^*]\rightarrow \mathbb{C}[G^*]$ coincides with the specialization of the map $T_i:\U_\A\rightarrow \U_\A$ under the isomorphism \eqref{eq:icl}. For any $w\in W$ with a reduced expression $w=s_{i_1}\cdots s_{i_r}$, define $\sigma_w = \sigma_{i_1}\cdots \sigma_{i_r}$. The map $\sigma_w$ is independent of the choice of reduced expressions for $w$ since $\sigma_i$ satisfy braid relations in $W$.

\subsection{Poisson structures on $K^\perp\backslash G^*$}\label{sec:psk}
We next consider Poisson structures associated with involutions. Recall the Satake diagram and parameters $c_i\in\{\pm1\}$ ($i\in\I_\circ$) from Section \ref{sec:iqg}.Then there is a unique algebraic group involution $\theta$ of $G$, such that 
\begin{align}\label{eq:ty}
\begin{split}
\theta(y_i(\xi))=\dot{w}_\bullet\cdot x_{\tau i}(c_i\xi)\cdot \dot{w}_\bullet^{-1},\qquad &\text{for }i\in\I_\circ, \;\xi\in\C,\\
\theta(y_j(\xi))=y_j(\xi),\quad \theta(x_j(\xi))=x_j(\xi),\qquad & \text{for }j\in\I_\bullet,\;\xi\in\mathbb{C}.
\end{split}
\end{align}

Let $K=G^\theta$ be the fixed-point subgroup, and 
$$
G_\theta=\{(g,\theta(g))\mid g\in G\}
$$
be the subgroup of $D$. Let $K^\perp$ be the identity component of $G^*\cap G_\theta$. Let $\X=K^\perp\backslash G^*$, which is equipped with the quotient manifold structure. By \cite[Proposition 3.1]{So24}, $K^\perp$ is a coisotropic subgroup of $G^*$. Hence there is a unique Poisson structure $\Pi_\X$ on $\X$ such that the canonical projection $(G^*,\Pi_{G^*})\rightarrow (\X,\Pi_\X)$ is Poisson. Let $\mathbb{C}[\X]$ be the ring of regular functions on $\X$ which carries a Poisson bracket induced by $\Pi_\X$. We recall the main theorem of \cite{So24} which describes the semiclassical limit of $\imath$quantum groups.

\begin{proposition}{\em \cite[Theorem 1]{So24}}\label{prop:iphi}
    There is a canonical isomorphism as Poisson algebras 
    \begin{equation}\label{eq:iphi}
    \varphi^\imath:{}\C\otimes_{\A'}\U^\imath_{\A'}\overset{\sim}{\longrightarrow}{}\C[\X],
\end{equation} 
such that the diagram
    \begin{equation}
        \begin{tikzcd}
            & \C[\X] \arrow[r,hook] & \C[G^*] \\
            & \mathbb{C}\otimes_{\A'}\U^\imath_{\A'} \arrow[u,"\varphi^\imath"] \arrow["_\C\iota"',r] & {}\C\otimes_{\A'} \U_{\A'} \arrow[u,"\varphi"']
        \end{tikzcd}
    \end{equation}
    commutes.     
    Here $_\mathbb{C}\iota$ is the base change of the natural embedding $\iota:\Ui_{\A'}\rightarrow \U_{\A'}$.
\end{proposition}

\begin{example}\label{ex:diagonal}
    Let $\widetilde{G}=G\times G$, and $\omega:G\rightarrow G$ be the group involution such that $\omega(x_i(a))=y_i(-a)$, for $i\in\I$, $a\in\mathbb{C}$, and $\omega(h)=h^{-1}$, for $h\in H$. Let $\widetilde{\theta}:\widetilde{G}\rightarrow\widetilde{G}$ be the map given by $\widetilde{\theta}(g_1,g_2)=(\omega(g_2),\omega(g_1))$. We consider the Poisson homogeneous space associated to the involution $\widetilde{\theta}$. Let $\widetilde{B}^+=B^+\times B^+$, and $\widetilde{B}^-=B^-\times B^-$. Then 
    \[
    \widetilde{G}^*=\{(u_+t,v_+h,t^{-1}u_-,h^{-1}v_-)\mid u_+,v_+\in U^+; u_-,v_-\in U^-; t,h\in H\}\subset \widetilde{B}^+\times \widetilde{B}^-,
    \]
    and \[
    \widetilde{K}^\perp=\{(u_+t,v_+t,\omega(v_+)t^{-1},\omega(u_+)t^{-1})\mid u_+,v_+\in U^+;t\in H\}.
    \]
    One has an isomorphism 
    \[
    \widetilde{K}^\perp\backslash\widetilde{G}^*\overset{\sim}{\longrightarrow} G^*
    \]
    as $\widetilde{G}^*$-Poisson homogeneous spaces, given by
    \[
    \widetilde{K}^\perp \cdot (u_+t,v_+h,t^{-1}u_-,h^{-1}v_-)\mapsto (\omega(v_-^{-1})h^{-1}u_+t,h\omega(v_+^{-1})t^{-1}u_-),
    \]
    where the right $\widetilde{G}^*$-action on $G^*$ is given by 
    \[
    (w_+s,s^{-1}w_-)\cdot (u_+t,v_+h,t^{-1}u_-,h^{-1}v_-)=(\omega(v_-^{-1})h^{-1}w_+su_+t, h\omega(v_+^{-1})s^{-1}w_-t^{-1}u_-^{-1}).
    \]
    Note that the action map is Poisson since $G^*$ is a Poisson-Lie group, and the morphism $G^*\rightarrow G^*$, given by $(v_+h,h^{-1}v_-)\mapsto (\omega(v_-^{-1})h^{-1},h\omega(v_+^{-1}))$, is Poisson.
\end{example}


\section{Integrality of relative braid group symmetries and PBW bases}\label{sec:intb}
In this section, we show that the relative braid group action preserves the DCKP-type integral form $\Ui_{\A'}$, and this leads to a (rescaled) PBW basis for $\Ui_{\A'}$. The key observation is that the DCKP-type integral form $\Ui_{\A'}$ is generated by Chevalley-type generators via rescaled $q$-commutators. 

\subsection{Rescaled $q$-commutators}\label{sec:qcom}

For $A, B$ in a $\Q(q^{1/2})$-algebra and $a\in\mathbb{Z}$, we write $[A,B]_{q^a} =AB -q^aBA$ and $\{A,B\}_{q^a} =\frac{[A,B]_{q^a}}{q-1}$ to be the \emph{$q$-commutators} and \emph{rescaled $q$-commutators}.

The following lemma follows from \eqref{eq:Ucom} and immediate computation.

\begin{lemma}\label{lem:qcom}
For any $A,B\in\U^+_{\A'}$ (resp., $\U^-_{\A'}$) and $a\in\Z$, we have $\{A,B\}_{q^a}\in \U^+_{\A'}$ (resp., $\U^-_{\A'}$).
\end{lemma}

The following proposition gives an intrinsic characterization of the integral forms $\U^\pm_{\A'}$ without using Lusztig's braid group symmetries.

\begin{proposition}\label{prop:gc}
    The $\A'$-subalgebra $\U^-_{\A'}$ (resp., $\U^+_{\A'}$) is the smallest $\A'$-subalgebra of $\U^+$ (resp., $\U^-$) which contains elements $F_i$ (resp., $E_i$) for $i\in\I$ and is closed under taking the rescaled $q$-commutator $\{\cdot,\cdot\}_{q^a}$ for any $a\in\N$.
\end{proposition}

\begin{proof}
    We prove the proposition for $\U^-_{\A'}$. The proof for $\U^+_{\A'}$ is similar and hence omitted.

    Let $\mathcal{U}$ be the smallest $\A'$-subalgebra of $\U^-$ which contains elements $F_i$ for $i\in\I$ and is closed under taking the rescaled $q$-commutators. Then $\mathcal{U}\subset\U^-_{\A'}$ by Lemma \ref{lem:qcom}.

    We next show that $\U^-_{\A'}\subset\mathcal{U}$. By the PBW basis of $\U^{-}_{\A'}$, it suffices to show that each root vector $F_{\beta_k}$ belong to $\mathcal{U}$. Recall the rescaling $\bF_i=\frac{1}{q_i^{1/2}(q_{i}-q_{i}^{-1})}F_i$ for $i\in\I$. By \cite[Theorem 3.5]{Su23}, the element $\bF_{\beta_k}=\frac{1}{q_{i_k}^{1/2}(q_{i_k}-q_{i_k})^{-1}}F_{\beta_k}$ can be written as $cM$, where $c\in\A'$ and $M$ is obtained via finite applications of $q$-commutators $[\cdot,\cdot]_{q^a}$ on the elements $\bF_i$ for various $i\in\I$ and $a\in\mathbb{Z}$. We remark that although \cite[Theorem 3.5]{Su23} only states that $c$ belongs to $\Q(q^{1/2})$, it follows from the proof that $c$ moreover belongs to $\A'$, because to obtain the desired expression, one only needs to apply \cite[(3.1) and Lemma 3.2]{Su23} repeatedly and the only denominator shown in the equation is $[-a_{ij}]_i!$ for $i\neq j\in\I$ which is easily checked to be invertible invertible in $\A'$. Let us write $M$ in terms of $q$-commutators. Suppose $M$ is obtained by applying $q$-commutators to elements $\bF_{j_1}$, $\mathbf{F}_{j_2}$, $\cdots$, $\mathbf{F}_{j_m}$ in an appropriate order. By rewriting each $q$-commutator $[\cdot,\cdot]_{q^a}$ to the rescaled $q$-commutator $\{\cdot,\cdot\}_{q^a}$ and writing each $\mathbf{F}_j$ in terms of $\frac{1}{q_j^{1/2}(q_j-q_j^{-1})}F_j$, we can write $M$ as $c'M'$, where $M'$ is obtained via $m-1$ applications of rescaled $q$-commutators $\{\cdot,\cdot\}_{q^a}$ on the elements $F_{j_1}$, $F_{j_2}$, $\cdots$, $F_{j_m}$ in the same order, and $c'=\frac{(q-1)^{m-1}}{\prod_{t=1}^m(q_{j_t}-q_{j_t}^{-1})}$. In conclusion, we have 
    $$F_{\beta_k}=c\frac{(q-1)^{m-1}(q_{i_k}-q_{i_k}^{-1})}{\prod_{t=1}^m q_{j_t}^{1/2}(q_{j_t}-q_{j_t}^{-1})}M',$$
    where $M'\in\mathcal{U}$ and $c\in\A'$. By the direct computation we have
    \[
    q_j-q_j^{-1}=q^{-1}[\varepsilon_j](q-1)(q+1).
    \]
     Hence
    \[
   \frac{(q-1)^{m-1}(q_{i_k}-q_{i_k}^{-1})}{\prod_{t=1}^m(q_{j_t}-q_{j_t}^{-1})}=\frac{q^{m-1}}{(q+1)^{m-1}}\frac{[ \varepsilon_{i_k}]}{\prod_{t=1}^m[\varepsilon_{j_t}]}\in\A'.
    \]
    Hence we conclude that $F_{\beta_k}\in\mathcal{U}$ which completes the proof.
\end{proof}

By Proposition~\ref{prop:Ucom}, $\Ui_{\A'}$ satisfy the condition \eqref{eq:com}. The following lemma is immediate from Lemma~\ref{lem:qcom}.

\begin{lemma}\label{lem:qcomi}
For any $A,B\in\U^\imath_{\A'}$ and $a\in\Z$, we have $\{A,B\}_{q^a}\in \U^\imath_{\A'}$.
\end{lemma}

\begin{proof}
By the direct computation, we have
\[
 \{A,B\}_{q^a}=\frac{AB-q^aBA}{q-1}=\frac{AB-BA}{q-1}-\frac{q^a-1}{q-1}BA.
\]
The lemma then follows from the condition \eqref{eq:com}.
\end{proof}

\subsection{Integrality of the projection map}
Recall $\pi^\imath_{\A'}$ from \eqref{eq:pia}.

\begin{proposition}\label{prop:intp}
    The map $\pi^\imath_{\A'}: \Ui_{\A'}\rightarrow \U_{P,\A'}$ is an isomorphism as $\A'$-modules.
\end{proposition}

\begin{proof}
Since $\pi^\imath:\Ui\rightarrow\U_P$ is a linear isomorphism, the map $\pi_{\A'}^\imath$ is injective. It suffices to show that $\pi^\imath_{\A'}$ is surjective.

We define a filtration on $\U$ by setting $\deg F_i=\deg E_i =1, \deg K_i=0, \deg E_j=\deg F_j=\deg K_j=0$ for $i\in \wI,j\in\bI$. This induces a filtration on $\Ui$ and a filtration on $\U_{\A'}$, and hence a filtration on $\Ui_{\A'}$ such that $\deg B_i=1$ for $i\in \wI$. By \cite[\S 2.4-2.5]{KY21}, the associated graded map $\gr (\pi^\imath): \gr(\Ui) \rightarrow \gr(\U_P)\cong\U_P$ is an algebra isomorphism, which sends $B_i\mapsto F_i$. By the above construction, $\gr (\pi^\imath_{\A'}): \gr(\Ui_{\A'}) \rightarrow \U_{P,\A'}$ is the restriction of $\gr(\pi^\imath)$ on $\gr(\Ui_{\A'})$ and hence $\gr (\pi^\imath_{\A'})$ is an injective algebra homomorphism.

Let $\Ui_{\A',\le n}\subset \Ui_{\A'}$ be the $\A'$-submodule spanning by elements with degree $\le n$. Let $A\in \Ui_{\A',\le n}, B\in \Ui_{\A',\le m}$ and $a\in \Z$. By Lemma \ref{lem:qcomi}, we have $\{A,B\}_{q^a}\in \Ui_{\A',\le m+n}$. Thus, $\gr(\Ui_{\A'})$ is closed under the rescaled $q$-commutators $\{\cdot, \cdot\}_{q^a}$.

We show that $\gr (\pi^\imath_{\A'})$ is surjective.  It follows from the definition that $\pi_{\A'}^\imath$ is a $\U_{\A'}^+(w_\bullet)\U_{\A'}^{\imath0}$-module homomorphism. Here we view $\U^\imath_{\A'}$ and $\U_{P,\A'}$ as $\U_{\A'}^+(w_\bullet)\U_{\A'}^{\imath0}$-modules via the left multiplication. Since any element in $\U_{\A'}^+(w_\bullet)\U_{\A'}^{\imath0}$ has degree $0$, $\gr (\pi_{\A'}^\imath)$ is also a $\U_{\A'}^+(w_\bullet)\U_{\A'}^{\imath0}$-module homomorphism. By \eqref{eq:tripA}, it suffices to show that $\U^-_{\A'}$ is in the image of $\gr (\pi^\imath_{\A'})$. Since $\gr (\pi^\imath)$ is an algebra homomorphism and $\gr(\Ui_{\A'})$ is closed under taking rescaled $q$-commutator, we conclude the image of $\gr (\pi^\imath)$ is an $\A'$-subalgebra which is closed under $q$-commutators. It moreover contains elements $F_i$ for $i\in \I$ by the definition. Hence it contains $\U^-_{\A'}$ by Proposition \ref{prop:gc}.

Therefore, $\gr \pi^\imath_{\A'}$ is an algebra isomorphism and then $\pi^\imath_{\A'}$ is a linear isomorphism.
\end{proof}

\begin{proposition}\label{prop:gci}
    The $\A'$-subalgebra $\Ui_{\A'}$ is the smallest $\A'$-subalgebra of $\Ui$ which contains elements $B_i$, $k_i^{\pm1}$ ($i\in\I_\circ$), $F_j$, $E_j$, $K_j^{\pm1}$ ($j\in\I_\bullet$), and is closed under taking the rescaled $q$-commutator $\{\cdot,\cdot\}_{q^a}$.
\end{proposition}

\begin{proof}
    Let $\mathcal{U}^\imath\subset \U^\imath$ be the $\A'$-subalgebra described as in the proposition. Then it is clear that $\mathcal{U}^\imath\subset \U^\imath_{\A'}$. The filtration on $\Ui$ induces a filtration on $\mathcal{U}^\imath$, and $\gr(\mathcal{U}^\imath)\subset \gr(\Ui_{\A'})$. Moreover it is clear that $\gr(\pi^\imath_{\A'})(\gr (\mathcal{U}^\imath))\subset \U_{P,\A'}$ is an $\A'$-subalgebra, closed under rescaled $q$-commutators, and contains  $\U_{\A'}^+(w_\bullet)\U_{\A'}^{\imath0}$ and $F_i$ for $i\in \I$. Therefore it follows from \eqref{eq:tripA} and Proposition \ref{prop:gc} that $\mathcal{U}^\imath=\Ui_{\A'}$.
\end{proof}

\subsection{Integrality of braid group symmetries}

 \begin{table}[H]
\caption{Rank 2 Satake diagrams}
     \label{table:rktwo}
 \resizebox{5.4 in}{!}{%
\begin{tabular}{|c| c|| c|c|}
\hline
\begin{tikzpicture}[baseline=0]
\node at (0, -0.15) {SP};
\end{tikzpicture}
&
\begin{tikzpicture}[baseline=0]
\node at (0, -0.15) {Satake diagrams};
\end{tikzpicture}
&
\begin{tikzpicture}[baseline=0]
\node at (0, -0.15) {SP};
\end{tikzpicture}
&
\begin{tikzpicture}[baseline=0]
\node at (0, -0.15) {Satake diagrams};
\end{tikzpicture}
\\
\hline
AI$_2$
&
\begin{tikzpicture}[baseline=0, scale=1.2]
		\node at (-0.5,0.2) {$\circ$};
		\node at (0.5,0.2) {$\circ$};
       \draw[-]  (0.45, 0.2) to (-0.45, 0.2);
		\node at (-0.5, 0) {\small 1};
		\node at (0.5, 0) {\small 2};
	\end{tikzpicture}
&
CII$_n$
&
\begin{tikzpicture}[baseline=6,scale=1.1]
		\node  at (0,0.2) {$\bullet$};
		\node  at (0,0) {1};
		\draw (0.05, 0.2) to (0.45, 0.2);
		\node  at (0.5,0.2) {$\circ$};
		\node  at (0.5,0) {2};
		\draw (0.55, 0.2) to (0.95, 0.2);
		\node at (1,0.2) {$\bullet$};
		\node at (1,0) {3};
		\node at (1.5,0.2) {$\circ$};
		\node at (1.5,0) {4};
		\draw[-] (1.05,0.2)  to (1.45,0.2);
		\draw[-] (1.55,0.2) to (1.95, 0.2);
		\node at (2,0.2) {$\bullet$};
		\node at (2,0) {5};
		\draw (1.9, 0.2) to (2.1, 0.2);
		\draw[dashed] (2.1,0.2) to (2.7,0.2);
		\draw[-] (2.7,0.2) to (2.9, 0.2);
		\node at (3,0.2) {$\bullet$};
		\draw[implies-, double equal sign distance]  (3.1,0.2) to (3.7, 0.2);
		\node at (3.8,0.2) {$\bullet$};
		\node at (3.8,0) {$n$};
	\end{tikzpicture}
\\
\hline
CI$_2$
&
\begin{tikzpicture}[baseline=0, scale=1.2]
		\node at (-0.5,0.2) {$\circ$};
		\node at (0.5,0.2) {$\circ$};
		\draw[-implies, double equal sign distance]  (0.4, 0.2) to (-0.4, 0.2);
		\node at (-0.5,0) {\small 1};
		\node at (0.5,0) {\small 2};
	\end{tikzpicture}
&
CII$_4$
&
\begin{tikzpicture}[baseline=6,scale=1.5]
        \node at (-1, 0.2) {$\bullet$};
        \node at (-1,0) {1};
		\draw[-] (-0.95,0.2) to (-0.55, 0.2);
        \node at (-0.5,0.2) {$\circ$};
        \node at (-0.5,0) {2};
		\draw[-] (-.45,0.2) to (-0.05, 0.2);
		\node at (0,0.2) {$\bullet$};
		\node at (0,0) {3};
		\draw[implies-, double equal sign distance]  (0.05, 0.2) to (0.75, 0.2);
		\node at (0.8,0.2) {$\circ$};
		\node at (0.8,0) {4};
	\end{tikzpicture}
\\
\hline
G$_2$
&
\begin{tikzpicture}[baseline=0, scale=1.5]
		\node at (-0.5,0) {$\circ$};
		\node at (0.5,0) {$\circ$};
		\draw[->]  (0.4, 0.05) to (-0.4, 0.05);
		\draw[->]  (0.4, -0.05) to (-0.4, -0.05);
		\draw[->]  (0.4, 0) to (-0.4, 0);
		\node at (-0.5, -.2) {\small 1};
		\node at (0.5,-.2) {\small 2};
	\end{tikzpicture}
&
EIV
&
\begin{tikzpicture}[baseline = 0, scale =1.5]
		\node at (-1,0.2) {$\circ$};
        \node at (-1,0) {1};
		\draw (-0.95,0.2) to (-0.55,0.2);
		\node at (-0.5,0.2) {$\bullet$};
        \node at (-0.5,0) {2};
		\draw (-0.45,0.2) to (-0.05,0.2);
		\node at (0,0.2) {$\bullet$};
        \node at (0.1,0) {3};
		\draw (0.05,0.2) to (0.45,0.2);
		\node at (0.5,0.2) {$\bullet$};
		\node at (0.5,0) {4};
		\draw (0.55,0.2) to (0.95,0.2);
		\node at (1,0.2) {$\circ$};
		\node at (1,0) {5};
		\draw (0, 0.15) to (0,-0.25);
		\node at (0,-0.2) {$\bullet$};
		\node at (-.15,-0.15) {6};
\end{tikzpicture}
\\
\hline
BI$_n$
&
 \begin{tikzpicture}[baseline=0, scale=1.2]
		\node at (0.5,0) {$\circ$};
		\node at (1.0,0) {$\circ$};
		\node at (1.5,0) {$\bullet$};
		\draw[-] (0.55,0)  to (0.95,0);
		\draw[-] (1.05,0)  to (1.45,0);
		\draw[-] (1.55,0) to (1.8, 0);
		\draw[dashed] (1.8,0) to (2.5,0);
		\draw[-] (2.5,0) to (2.75, 0);
		\node at (2.8,0) {$\bullet$};
		\draw[-implies, double equal sign distance]  (2.85, 0) to (3.45, 0);
		\node at (3.5,0) {$\bullet$};
		\node at (0.5,-.2) {\small 1};
		\node at (1,-.2) {\small 2};
		\node at (1.5,-.2) {\small 3};
		\node at (3.5,-.2) {$n$};
	\end{tikzpicture}
&
AIII$_3$
&
\begin{tikzpicture}[baseline=0,scale=1.5]
		\node  at (-0.65,0) {$\circ$};
		\node  at (0,0) {$\circ$};
		\node  at (0.65,0) {$\circ$};
		\draw[-] (-0.6,0) to (-0.05, 0);
		\draw[-] (0.05, 0) to (0.6,0);
		\node at (-0.65,-0.15) {1};
		\node at (0,-0.15) {2};
		\node at (0.65,-0.15) {3};
        \draw[bend left,<->,red] (-0.65,0.1) to (0.65,0.1);
        \node at (0,0.2) {$\textcolor{red}{\tau}$};
	\end{tikzpicture}
\\
\hline
DI$_n$
&
  \begin{tikzpicture}[baseline=0, scale=1.2]
		\node at (0.55,0) {$\circ$};
		\node at (1.05,0) {$\circ$};
		\node at (1.5,0) {$\bullet$};
		\draw[-] (0.6,0)  to (1.0,0);
		\draw[-] (1.1,0)  to (1.4,0);
		\draw[-] (1.4,0) to (1.9, 0);
		\draw[dashed] (1.9,0) to (2.7,0);
		\draw[-] (2.7,0) to (2.9, 0);
		\node at (3,0) {$\bullet$};
		\node at (3.8,0.35) {$\bullet$};
		\node at (3.8,-0.35) {$\bullet$};
        \draw (3,0) to (3.8,0.35);
        \draw (3,0) to (3.8,-0.35);
		\node at (0.5,-.2) {\small 1};
		\node at (1,-.2) {\small 2};
		\node at (1.5,-.2) {\small 3};
	\end{tikzpicture}
&
AIII$_n$
&
 \begin{tikzpicture}[baseline=0,scale=1.0]
		\node  at (-2.1,0) {$\circ$};
		\node  at (-1.3,0) {$\circ$};
		\node  at (-0.5,0) {$\bullet$};
		\node  at (0.5,0) {$\bullet$};
		\node  at (1.3,0) {$\circ$};
		\node  at (2.1,0) {$\circ$};
		\draw[-] (-2.05,0) to (-1.35, 0);
		\draw[-] (-1.25,0) to (-0.55, 0);
		\draw[-] (0.55,0) to (1.25, 0);
		\draw[-] (1.35, 0) to (2.05,0);
		\node at (-2.1,-0.2) {1};
		\node at (-1.3,-0.2) {2};
		\node at (1.3,-0.2) {\small$n-1$};
		\node at (2.1,-0.2) {$n$ };
        \draw[dashed] (-0.5,0) to (0.5,0);
        \draw[bend left,<->,red] (-1.3,0.1) to (1.3,0.1);
        \draw[bend left,<->,red] (-2.1,0.1) to (2.1,0.1);
        \node at (0,0.6) {$\textcolor{red}{\tau} $};
	\end{tikzpicture}
\\
\hline
DIII$_4$
&
\begin{tikzpicture}[baseline=0]
\end{tikzpicture}
  \begin{tikzpicture}[baseline=0, scale=1.2]
		\node at (0.65,0) {$\circ$};
		\node at (1.5,0) {$\circ$};
		\draw[-] (0.7,0)  to (1.45,0);
		\node at (2.3,0.4) {$\bullet$};
		\node at (2.3,-0.4) {$\bullet$};
        \draw (1.55,0) to (2.3,0.4);
        \draw (1.55,0) to (2.3,-0.4);
		\node at (0.65,-.2) {\small 1};
		\node at (1.5,-.2) {\small 2};
		\node at (2.35,0.25) {\small 3};
		\node at (2.35,-0.25) {\small 4};
	\end{tikzpicture}
&
DIII$_5$
&
\begin{tikzpicture}[baseline=0,scale=1.5]
		\node at (-0.5,0) {$\bullet$};
		\node at (0,0) {$\circ$};
		\node at (0.5,0) {$\bullet$};
        \node at (1,0.3) {$\circ$};
        \node at (1,-0.3) {$\circ$};
		\draw[-] (-0.45,0) to (-0.05, 0);
		\draw[-] (0.05, 0) to (0.45,0);
		\draw[-] (0.5,0) to (0.965, 0.285);
		\draw[-] (0.5, 0) to (0.965,-0.285);
		\node at (-0.5,-0.2) {1};
		\node at (0,-0.2) {2};
		\node at (0.5,-0.2) {3};
        \node at (1,0.15) {4};
        \node at (1,-0.45) {5};
        \draw[bend left,<->,red] (1.1,0.3) to (1.1,-0.3);
        \node at (1.4,0) {$\textcolor{red}{\tau} $};
	\end{tikzpicture}
\\
\hline
AII$_5$
&
 \begin{tikzpicture}[baseline=0,scale=1.2]
		\node at (-0.5,0) {$\bullet$};
		\node  at (0,0) {$\circ$};
		\node at (0.5,0) {$\bullet$};
		\node at (1,0) {$\circ$};
		\node at (1.5,0) {$\bullet$};
		\draw[-] (-0.5,0) to (-0.05, 0);
		\draw[-] (0.05, 0) to (0.5,0);
		\draw[-] (0.5,0) to (0.95,0);
		\draw[-] (1.05,0)  to (1.5,0);
		\node at (-0.5,-0.2) {1};
		\node  at (0,-0.2) {2};
		\node at (0.5,-0.2) {3};
		\node at (1,-0.2) {4};
		\node at (1.5,-0.2) {5};
	\end{tikzpicture}
&
EIII
&
\begin{tikzpicture}[baseline = 0, scale =1.5]
		\node at (-1,0) {$\circ$};
        \node at (-1,-0.2) {1};
		\draw (-0.95,0) to (-0.55,0);
		\node at (-0.5,0) {$\bullet$};
        \node at (-0.5,-0.2) {2};
		\draw (-0.45,0) to (-0.05,0);
		\node at (0,0) {$\bullet$};
        \node at (0.1,-0.2) {3};
		\draw (0.05,0) to (0.45,0);
		\node at (0.5,0) {$\bullet$};
		\node at (0.5,-0.2) {4};
		\draw (0.55,0) to (0.95,0);
		\node at (1,0) {$\circ$};
		\node at (1,-0.2) {5};
		\draw (0,-0.05) to (0,-0.35);
		\node at (0,-0.4) {$\circ$};
		\node at (-.15,-0.35) {6};
        \draw[bend left, <->, red] (-0.9,0.1) to (0.9,0.1);
        \node at (0,0.2) {$\color{red} \tau $};
	\end{tikzpicture}
\\
\hline
\end{tabular}
}%
\end{table}

We list all rank $2$ Satake diagrams in Table~\ref{table:rktwo}; see also \cite{WZ23,DK19}. 

\begin{lemma}\label{lem:braid-rk2}
Let $(\I=\bI\cup\wI,\tau)$ be a Satake diagram of rank two and $(\U,\Ui)$ be the associated quantum symmetric pair. For any $i,j\in \wI$ such that $j\neq i,\tau i$, we have $\TT_i(B_j)\in \Ui_{\A'}$.
\end{lemma}

\begin{proof}
The rank two formulas of relative braid group symmetries on the universal $\imath$quantum group were given in \cite[Table 4]{WZ23}, and then the formulas for $\TT_i(B_j),j\neq i,\tau i$ are obtained by taking the central reduction and applying the isomorphism $\Omega$ in Remark~\ref{rmk:bUi}. For each rank $2$ Satake diagram in Table~\ref{table:rktwo}, we provide below the closed formulas for $\TT_i(B_j)$ obtained in this way.
\begin{itemize}
\item[$\boxed{\text{AI}_1}$] We have $\TT_2(B_1)=\frac{[B_1,B_2]_q}{q^{1/2}(q-q^{-1})}=\frac{\{B_1,B_2\}_q}{q^{1/2}+q^{-1/2}}$. 

\item[$\boxed{\text{CI}_2}$] We have 
\[
\TT_1(B_2)=\frac{1}{q[2](q-q^{-1})^2}\big[[B_2,B_1]_{q^2},B_1\big]+B_2
=\frac{\big\{\{B_2,B_1\}_{q^2},B_1\big\}}{[2](q^{1/2}+q^{-1/2})^2}+B_2.
\]

\item[$\boxed{\text{G}_2}$] We have
\begin{align*}
\TT_{1}(B_2) 
& =\frac{\Big\{\big\{ \{B_2, B_1\}_{q^3},B_1 \big\}_{q}, B_1\Big\}_{q^{-1}}}{[3]! (q^{1/2}+q^{-1/2})^3} 
-\frac{q[2]^2\{ B_2, B_1\}_{q^3} + q^3[3] \{B_2, B_1\}_{q^{-1}}}{[3]!(q^{1/2}+q^{-1/2})}.
\end{align*}

\item[$\boxed{\text{BI}_n, \text{DI}_n, \text{DIII}_4}$] We have 
\begin{align*}
\TT_2(B_1)&=\frac{1}{q_2 (q_2-q_2^{-1})^2}\Big[[B_1,B_2]_{q_2},T_{\bw}^{-1}(B_2)\Big]_{q_2} +B_2 K_{\alpha_2-\bw(\alpha_2)}
\\
&=\frac{1}{(q_2^{1/2}+q_2^{-1/2})^2}\Big\{\{B_1,B_2\}_{q_2},T_{\bw}^{-1}(B_2)\Big\}_{q_2} +B_2 K_{\alpha_2-\bw(\alpha_2)}.
\end{align*}



\item[$\boxed{\text{AII}_5}$] We have
$
\TT_4(B_2)=\frac{q^{-1/2}}{q-q^{-1}}[B_2,T_3^{-1}(B_4)]_q=\frac{1}{q^{1/2}+q^{-1/2}}\{B_2,T_3^{-1}(B_4)\}_q.
$

\item[$\boxed{\text{CII}_n}$] We have
\begin{align*}
\TT_4(B_2)&=\frac{1}{q(q-q^{-1})^2}\Big[B_2, \big[T_3^{-1}(B_4), T_{5\cdots n\cdots 5}^{-1}(B_4)\big]_q\Big]_{q}
+q^{3/2} T_3^2(B_2) K_{s_3(\alpha_4)-s_3\bw(\alpha_4)}
\\
&=\frac{\Big\{B_2, \big\{T_3^{-1}(B_4), T_{5\cdots n\cdots 5}^{-1}(B_4)\big\}_q\Big\}_{q}}{(q^{1/2}+q^{-1/2})^2}
+q^{3/2} T_3^2(B_2) K_{s_3(\alpha_4)-s_3\bw(\alpha_4)}.
\end{align*}

\item[$\boxed{\text{CII}_4}$] We have
\begin{align*}
\TT_2(B_4)&=\frac{1}{q[2](q-q^{-1})^2}\Big[\big[B_4,T_3^{-1}(B_2)\big]_{q^2}, T_3^{-1}(B_2)\Big]-\frac{q^{-1}[B_4,F_3]_{q^2} E_1 K_{3}}{[2](q-q^{-1})}
\\
&=\frac{1}{[2](q^{1/2}+q^{-1/2})^2}\Big\{\big\{B_4,T_3^{-1}(B_2)\big\}_{q^2}, T_3^{-1}(B_2)\Big\}-\frac{q^{-1/2}\{B_4,F_3\}_{q^2} E_1 K_{3}}{[2](q^{1/2}+q^{-1/2})}.
\\
\TT_4(B_2)&=\frac{1}{q^{3/2}(q-q^{-1})(q^2-q^{-2})}\big[B_2,[F_3,B_4]_{q^2}\big]_q
=\frac{\big\{B_2,\{F_3,B_4\}_{q^2}\big\}_q}{(q^{1/2}+q^{-1/2})(q+q^{-1})}.
\end{align*}
(For type CII$_4$, an overall factor of $\frac{1}{[2]}$ was missing in the formula for $\TT_2(B_4)$ in \cite[Table 3-4]{WZ23}. The formula here is obtained from the correct formula.)

\item[$\boxed{\text{EIV}}$] We have
\begin{align*}
\TT_1(B_5)=\frac{1}{q^{1/2}(q-q^{-1})}\Big[B_5,T_{234}^{-1}(B_1)\Big]_q=\frac{1}{q^{1/2}+q^{-1/2}}\Big\{B_5,T_{234}^{-1}(B_1)\Big\}_q.
\end{align*}

\item[$\boxed{\text{AIII}_3}$] We have 
$\TT_{1}(B_2)=\frac{1}{q(q-q^{-1})^2}\big[[B_2,B_1]_q,B_3\big]_q+B_2=\frac{\big\{\{B_2,B_1\}_q,B_3\big\}_q}{(q^{1/2}+q^{-1/2})^2}+B_2$. 

\item[$\boxed{\text{AIII}_n}$] In this case, by \eqref{def:ci}, $c_2=(-1)^a,c_{n-1}=(-1)^{b},a+b=4-n$. We have
\begin{align*}
\TT_2(B_1)&=\frac{(-1)^{n/2+1}}{q(q-q^{-1})^2} \Big[[B_1,B_2]_q,T_{\bw}^{-1}(B_{n-1})\Big]_q+(-1)^{n/2+b}q^{-1/2} K_{h_2-\bw(h_{n-1})} B_1
\\
&=\frac{(-1)^{n/2+1}}{(q^{1/2}+q^{-1/2})^2} \Big\{\{B_1,B_2\}_q,T_{\bw}^{-1}(B_{n-1})\Big\}_q
+(-1)^{n/2+b}q^{-1/2} K_{h_2-\bw(h_{n-1})} B_1.
\end{align*}

\item[$\boxed{\text{DIII}_5}$] In this case, by \eqref{def:ci}, $c_4=(-1)^a,c_{5}=(-1)^{b},a+b=-1$. We have
\begin{align*}
\TT_4(B_2)&=\frac{(-1)^{3/2}}{q(q-q^{-1})^2}\Big[[B_2,T_3^{-1}(B_5)]_q,B_4\Big]_q+(-1)^{1/2+a}q^{1/2} T_3^2(B_2) K_{s_3(h_5)-s_3\bw(h_4)}
\\
&=\frac{(-1)^{3/2}}{(q^{1/2}+q^{-1/2})^2}\Big\{\big\{B_2,T_3^{-1}(B_5)\big\}_q,B_4\Big\}_q
\\
&\quad +(-1)^{1/2+a}q^{1/2} T_3^2(B_2) K_{s_3(h_5)-s_3\bw(h_4)}.
\end{align*}

\item[$\boxed{\text{EIII}}$] In this case, by \eqref{def:ci}, $c_1=(-1)^a,c_{5}=(-1)^{b},a+b=-3$. We have
\begin{align*}
\TT_1(B_6)&=\frac{(-1)^{5/2}}{q(q-q^{-1})^2} \Big[[B_6, T_{23}^{-1}(B_1)]_q, T_4^{-1}(B_5)\Big]_q
\\
&\quad+(-1)^{3/2+b}q^{1/2} T_{32323}(B_6) K_{s_4\bw(h_1)-s_4(h_5)}
\\
&=(-1)^{5/2}\frac{\Big\{\{B_6, T_{23}^{-1}(B_1)\}_q, T_4^{-1}(B_5)\Big\}_q}{(q^{1/2}+q^{-1/2})^2} 
\\
&\quad+(-1)^{3/2+b}q^{1/2} T_{32323}(B_6) K_{s_4\bw(h_1)-s_4(h_5)},
\\
\TT_6(B_1)&=\frac{1}{q^{1/2}(q-q^{-1})}\big[B_1,T_{32}^{-1}(B_6)\big]_q=\frac{1}{q^{1/2}+q^{-1/2}}\big\{B_1,T_{32}^{-1}(B_6)\big\}_q.
\end{align*} 
\end{itemize}
By Lemma~\ref{lem:qcomi} and Proposition~\ref{prop:Tj}, it is clear from the above rank two formulas that $\TT_i(B_j)\in \Ui_{\A'}$ for $j\neq i,\tau i$.
\end{proof}

\begin{theorem}\label{thm:intt}
    For any $i\in\I_\circ$ and $x\in\Ui_{\A'}$, we have $\TT_i(x),\TT_i^{-1}(x)\in\Ui_{\A'}$.
\end{theorem}

\begin{proof}
    We first show the statement for $\TT_i$. Since $\TT_i$ is an algebra automorphism, it preserves the rescaled $q$-commutator $\{\cdot,\cdot\}_{q^a}$ for any $a\in \Z$. By Lemma \ref{lem:qcomi}, the subspace 
    \[
    \{x\in\Ui\mid \TT_i(x)\in\Ui_{\A'}\}\subset\Ui
    \]
    is an $\A'$-subalgebra of $\Ui$ and closed under the rescaled $q$-commutator. Thanks to Proposition \ref{prop:gci}, it suffices to show that $\TT_i(x)\in\Ui_{\A'}$ when $x$ lies in
    \begin{align}\label{eq:genUi}
    \big\{B_i, k_i^{\pm1} , F_j, E_j, K_j^{\pm1} |i\in \wI, j\in\I_\bullet\big\}.
    \end{align}

    If $x=E_j,F_j,K_j$ for $j\in \bI$, the formula of $\TT_i(x)$ is given in Proposition~\ref{prop:WZ}(2); if $x=k_p$ for $p\in \wI$, the formula of $\TT_i(x)$ is given in Proposition~\ref{prop:WZ}(3); if $x=B_i,B_{\tau i}$, the formula of $\TT_i(x)$ is given in Proposition~\ref{prop:WZ}(4)(5). It is clear from these formulas that $\TT_i(x)\in\Ui_{\A'}$. If $x=B_j,j\neq i,\tau i$, then $\TT_i(x)\in\Ui_{\A'}$ is proved in Lemma~\ref{lem:braid-rk2}. The proof is completed.

    We show the statement for $\TT_i^{-1}$. Recall from Lemma~\ref{lem:sigmaT} that $\TT_i^{-1}=\sigma_\tau \circ \TT_i \circ \sigma_{\tau}$. It suffices to show that $\sigma_\tau$ preserves $\Ui_{\A'}$. By Lemma \ref{lem:qcomi}, the subspace $\{x\in\Ui\mid \sigma_\tau(x)\in\Ui_{\A'}\}\subset\Ui$
    is an $\A'$-subalgebra of $\Ui$ and closed under the rescaled commutator $\{\cdot,\cdot\}_{q^a}$. By Proposition \ref{prop:gci}, it suffices to show that $\sigma_\tau(x)\in\Ui_{\A'}$ when $x$ lies in \eqref{eq:genUi}. This is clear from the definition \eqref{sigmatau} of $\sigma_\tau$.

    The proof is completed.
   \end{proof} 

\begin{remark}
It is expected that Theorem~\ref{thm:intt} remains valid over $\A=\Z[q^{\pm 1/2}]$ for quasi-split types. However, the localization $\A'$ is necessary in the current proof. The current ring $\A'$ is sufficient for the intended future applications. It is interesting to establish a new approach to prove Theorem~\ref{thm:intt} for $\Ui_\A$.
\end{remark}

\begin{proposition}\label{prop:piT}
For any $i,j\in \wI$ such that $j\neq i,\tau i$, we have
\begin{align}\label{eq:piT}
\pi^\imath(\TT_i(B_j))=\pi_P(T_{\bs_i}(B_j)).
\end{align}
\end{proposition}

\begin{proof}
By \cite[Proposition 6.5, Theorem 6.6]{WZ23} (and applying the central reduction), there exist a non-commutative polynomial $\cP(u_1,u_2,v_1,v_2,z)$, which is linear in $z$, such that 
\begin{align}\notag
\TT_i(B_j)&=\cP(B_i,B_{\tau i},k_i,k_{\tau i},B_j)
\\
&=\cP(B_i,B_{\tau i},k_i,k_{\tau i},F_j)+\cP(B_i,B_{\tau i},k_i,k_{\tau i},T_{\bw}(E_{\tau j})K_j^{-1}),
\\
T_{\bs_i}(B_j)&=\cP(B_i,B_{\tau i},k_i,k_{\tau i},F_j)+\cP(\hat{B}_i,\hat{B}_{\tau i},k_i,k_{\tau i},T_{\bw}(E_{\tau j})K_j^{-1}).
\end{align}
where $\hat{B}_i=T_{\bs_i}\circ \TT_i^{-1}(B_i)$.  Note that $\alpha_{\tau j}$ appears with multiplicity $-1$ in the weights of the following two terms
\[
\cP(B_i,B_{\tau i},k_i,k_{\tau i},T_{\bw}(E_{\tau j})K_j^{-1}),\qquad
\cP(\hat{B}_i,\hat{B}_{\tau i},k_i,k_{\tau i},T_{\bw}(E_{\tau j})K_j^{-1}).
\]
By the construction of $\pi_P$, these two terms are annihilated by $\pi_P$. Hence, the desired identity \eqref{eq:piT} follows.
\end{proof}

\subsection{PBW basis for $\Ui_{\A'}$}

We first construct a (rescaled) PBW basis for $\Ui$ following \cite{LYZ24}. By \cite[Lemma 2.3]{LYZ24}, $\cR^+$ is the disjoint union of $\cR^+(\bbw_0)$ and $\cR_{\bullet}^+=\cR^+(w_\bullet)$. For $\gamma\in \cR^+$, we defined the root vectors $F_\gamma, E_\gamma$ in the same way as \eqref{eq:F-root}, using a reduced expression of $w_\bullet$.

Let $\beta\in\cR^+(\bbw_0)$. We construct the root vectors $B_\beta$ in $\Ui$ as follows. Write $\beta=w(\alpha_{i_0})$ for some $w\in W,i_0\in \I$. 
Let $\mathbf{B}_\beta$ denote the root vector constructed in \cite[Section 3-4]{LYZ24}. We define
\begin{align}\label{def:Bbeta}
B_\beta:=q_{i_0}^{1/2}(q_{i_0}-q_{i_0}^{-1})\mathbf{B}_\beta.
\end{align}
By \cite[Proposition 4.1]{LYZ24}, $B_\beta=B_i$ if $\beta=\alpha_i,i\in \wI$.

Equivalently, $B_\beta$ can be precisely constructed as follows. When the underlying Satake diagram is rank 1 (i.e., $\bbw_0=\bs_i$ for some $i\in \wI$), $B_\beta$ are listed explicitly in Appendix~\ref{app:A}. In the higher rank case, fix a reduced expression $\bbw_0=\bs_{i_1} \cdots \bs_{i_{\ell}}$ in $W^\circ$. For any $\beta\in \cR^+(\bbw_0)$, there exist a unique $1\leq j\leq \ell$ and a unique $\beta_0\in \cR^+(\bs_j)$ such that $\beta=\bs_{i_1} \cdots \bs_{i_{j-1}} (\beta_0)$. By \cite[(4.1)]{LYZ24}, we have
\begin{align}\label{eq:Bbeta}
B_\beta=\TT_{i_1}\TT_{i_2}\cdots \TT_{i_{j-1}}(B_{\beta_0}).
\end{align}

\begin{proposition}\label{le:bau}
For any reduced expression of $\bbw_0$ in $W^\circ$ and any $\beta\in\mathcal{R}^+(\bbw_0)$, the element $B_\beta$ belongs to $\Ui_{\A'}$.
\end{proposition}

\begin{proof}
By Theorem \ref{thm:intt} and expression \eqref{eq:Bbeta}, it suffices to show that $B_{\beta}\in\Ui_{\A'}$ for $\beta\in\mathcal{R}^+(\mathbf{r}_i)$ with $i\in\I_\circ$. This is checked in Appendix \ref{app:A}.
\end{proof}

Following \cite[(4.2)]{LYZ24}, using the reduced expression of $\bs_{i_j}$ as Table \ref{table:localSatake}, We obtain a reduced expression of $\bbw_0$ in $W$, which induces a total order on $\mathcal{R}^+(\bbw_0)$. For any $\ba=(a_\beta)_{\beta\in\mathcal{R}^+(\bbw_0)}\in\N^{\ell(\bbw_0)}$, we set
\[
B^\ba=\prod_{\beta\in\mathcal{R}^+(\bbw_0)}B_\beta^{a_\beta}.
\]

Let $Y^\imath=\{\mu\in \Z\I|\theta\mu=\mu\}$. Clearly, $K_\mu$ for $\mu\in Y^\imath$ form a basis for $\U^{\imath0}_{\A'}$.

\begin{theorem}\label{thm:PBW}
    Let $\bbw_0=\bs_{i_1}\bs_{i_2}\cdots\bs_{i_\ell}$ be a reduced expression for the longest element $\bbw_0$ in $W^\circ$, and $\bw=s_{j_1}s_{j_2}\cdots s_{j_r}$ a reduced expression of the longest element $w_\bullet$ in $W_\bullet$. Then the monomials
$$B^\ba\cdot F_\bullet^{\bc} \cdot E_\bullet^{\bd} \cdot K_\mu \qquad (\ba\in \N^{\ell(\bbw_0)},\bc,\bd\in\N^{r},\mu\in Y^{\imath})$$
form an $\A'$-basis for $\Ui_{\A'}$.
\end{theorem}

\begin{proof}
By Proposition~\ref{le:bau}, the monomials $B^\ba F_\bullet^{\bc} E_\bullet^{\bd} K_\mu$ lies in $\Ui_{\A'}$. By \cite[Corollary 4.5]{LYZ24}, these monomials are linearly independent. It suffices to show that they form a spanning set of $\Ui_{\A'}$. Recall the algebra isomorphism $\gr(\pi_{\A'}^\imath):\gr\Ui_{\A'}\rightarrow \U_{P,\A'}$ from the proof of Proposition \ref{prop:intp}. By \cite[Lemma 4.2]{LYZ24}, $\gr(\pi_{\A'}^\imath)(B_\beta)=F_\beta$ and hence $\gr(\pi_{\A'}^\imath)(B^\ba F_\bullet^{\bc} E_\bullet^{\bd} K_\mu)=F^\ba F_\bullet^{\bc} E_\bullet^{\bd} K_\mu$. This implies that the images of monomials $B^\ba F_\bullet^{\bc} E_\bullet^{\bd} K_\mu$ under $\gr(\pi_{\A'}^\imath)$ form a PBW basis for $\U_{P,\A'}$. Hence, images of these monomials in $\gr\Ui_{\A'}$ form an $\A'$-basis for $\gr\Ui_{\A'}$ and then these monomials form an $\A'$-basis for $\Ui_{A'}$.
\end{proof}

\begin{remark}
We refer to the basis in Theorem~\ref{thm:PBW} as the {\em rescaled PBW basis} for $\Ui_{\A'}$, due to \eqref{def:Bbeta}. 
Under the embedding $\Ui_{\A'}\hookrightarrow \U_{\A'}$, each element of the rescaled PBW basis for $\Ui_{\A'}$ is an $\A'$-linear combination of the (rescaled) PBW basis for $\U_{\A'}$ in Proposition~\ref{prop:dec}. It would be interesting to investigate the more precise relation between the rescaled PBW basis for $\Ui_{\A'}$ and the PBW basis for $\U_{\A'}$.
\end{remark}

\begin{remark}
In \cite{LP25}, Lu and Pan constructed a rescaled Hall basis and the dual canonical basis for $\imath$quantum groups via a Hall algebra approach. We expect that the integral form spanned by the dual canonical basis of $\Ui$ is the same as $\Ui_{\A}$. 
\end{remark}

\subsection{Root vectors under the projection $\pi^\imath$} In this subsection, we compare the images of root vectors $B_{\beta}$ under the projection $\pi^{\imath}$ with root vectors $F_\beta$ in $\U_P$. 

\subsubsection{Type AI}
Let us consider the Satake diagram of type AI. 
\begin{equation*}
 \begin{tikzpicture}[baseline=0,scale=1.5]
		\node at (-0.5,0) {$\circ$};
		\node  at (0,0) {$\circ$};
		\node at (0.5,0) {$\cdots$};
		\node at (1,0) {$\circ$};
		\node at (1.5,0) {$\circ$};
		\draw[-] (-0.45,0) to (-0.05, 0);
		\draw[-] (0.05, 0) to (0.25,0);
		\draw[-] (0.75,0) to (0.95,0);
		\draw[-] (1.05,0)  to (1.45,0);
		\node at (-0.5,-0.2) {$1$};
		\node  at (0,-0.2) {$2$};
		\node at (1,-0.2) {$n-1$};
		\node at (1.5,-0.2) {$n$};
	\end{tikzpicture}
\end{equation*}
In this case $\bs_i=s_i$ for $i\in \I$. Fix the reduced expression $w_0=\epsilon_1 \epsilon_2 \cdots \epsilon_{\lfloor (n+1)/2\rfloor}$ where $\epsilon_i=s_i s_{i+1}\cdots s_{n-i} s_{n+1-i} s_{n-i} \cdots s_i$.

\begin{proposition}
Under the above reduced expression, we have $\pi^{\imath}(B_\beta)=F_{\beta}$ for any $\beta\in \cR^+$.
\end{proposition}

\begin{proof}
Let $1\le i\le \lfloor (n+1)/2\rfloor$. Note that $\epsilon_i$ commutes with $s_{j}$ for $i+1\leq j\leq n-i$. By \cite[Theorem 7.13]{WZ23}, we have $\TT_{\epsilon_i} B_j=B_j$ and $\TT_{\epsilon_i}\TT_j=\TT_j \TT_{\epsilon_i}$ for $i+1\leq j\leq n-i$. Hence, the root vectors $B_{\beta}$ can be simplified to exactly one of the following two formulas
\begin{align*}
&\TT_{i}\TT_{i+1}\cdots \TT_{j-1}(B_j), \qquad\qquad \qquad  (i\le j\le n+1-i),
\\
&\TT_{i}\TT_{i+1}\cdots \TT_{n+1-i}\TT_{n-i}\cdots \TT_{j+1}(B_j)
\\
&=\TT_{j+1}^{-1} \TT_{j+2}^{-1}\cdots \TT_{n-i}^{-1}(B_{n+1-i}), \qquad (i\le j\le n-i).
\end{align*}
By Lemma~\ref{lem:braid-rk2}, we have 
\[
\TT_{i}\TT_{i+1}\cdots \TT_{j-1}(B_j)=\Big[B_j,\big[B_{j-1}, \cdots [B_{i+1},B_i]_q\big]_q\Big]_q.
\]
Note that each $B_r,r\in \I$ appears at most once in the above formula. For the weight reason, we have
\[
\pi^\imath\TT_{i}\TT_{i+1}\cdots \TT_{j-1}(B_j)=\Big[F_j,\big[F_{j-1}, \cdots [F_{i+1},F_i]_q\big]_q\Big]_q=T_{i}T_{i+1}\cdots T_{j-1}(F_j).
\]
The formulas of $\TT_i^{-1}$ are obtained by Lemma~\ref{lem:sigmaT}. Using these formulas and similar arguments above, one can also show that $\pi^\imath$ sends $\TT_{j+1}^{-1} \TT_{j+2}^{-1}\cdots \TT_{n-i}^{-1}(B_{n+1-i})$ to $T_{j+1}^{-1} T_{j+2}^{-1}\cdots T_{n-i}^{-1}(F_{n+1-i})$ as desired.
\end{proof}

\subsubsection{Diagonal type} Let us consider a Satake diagram $(\widetilde{\I},\tau)$ of diagonal type where $\widetilde{\I}=\{i,i'| i\in \I\}$ and $\tau:i\leftrightarrow i'$. Let $\big(\U(\g\oplus \g),\Ui\big)$ be the associated quantum symmetric pair. Let $\g$ be the Lie algebra associated to $\I$. The relative Weyl group $\widetilde{W}^\circ$ is isomorphic to the Weyl group $W$ of $\g$ via $\bs_i\mapsto s_i (i\in \I)$.

In this setting, there is an algebra isomorphism $\Ui\rightarrow \U(\g)$ given by $B_i\mapsto F_i, B_{\tau i}\mapsto E_i, k_i\mapsto K_i, (i\in \I)$. Under this isomorphism, $\TT_i$ for $i\in \I$ is identified with Lusztig symmetry $T_i$; cf. \cite[Section 4]{WZ25}.

\begin{proposition}
Let $\tilde{w}_0=\bs_{i_1}\cdots \bs_{i_l}$ be a reduced expression of the longest element $\tilde{w}_0$ of $\widetilde{W}^\circ$. Let $B_\beta$ be the  associated root vectors in $\Ui$. We have $\pi^\imath(B_\beta)=F_{\beta}$ for any $\beta\in \widetilde{\cR}^+$. 
\end{proposition}

\begin{proof}
By definition, the root vectors $B_\beta$ are given by either one of the following formula
\begin{align*}
\TT_{i_1} \cdots \TT_{i_{k-1}} (B_{i_k}),\qquad \TT_{i_1} \cdots \TT_{i_{k-1}} (B_{i_k'}),
\qquad(1\le k\le l).
\end{align*}
Using the above isomorphism $\Ui\cong \U(\g)$, the formula of $\TT_{i_1} \cdots \TT_{i_{k-1}} (B_{i_k})$ is obtained by replacing $F_i$ by $B_i$ in the formula of $T_{i_1} \cdots T_{i_{k-1}} (F_{i_k})$. In particular, $B_{j'}$ for $j\in\I$ does not appear in the formula of $\TT_{i_1} \cdots \TT_{i_{k-1}} (B_{i_k})$. Recall that $B_i=F_i+E_{i'} K_i^{-1}$ and note that $F_i$ commutes with $E_{j'}$ for $i,j\in \I$. Hence, if we expand $\TT_{i_1} \cdots \TT_{i_{k-1}} (B_{i_k})$ in terms of weight homogeneous terms in $\U(\g\oplus \g)$, the only term with a weight in $\N\I$ is given by $T_{i_1} \cdots T_{i_{k-1}} (F_{i_k})$. Therefore, $\pi^\imath \TT_{i_1} \cdots \TT_{i_{k-1}} (B_{i_k})$ equals $\T_{i_1} \cdots \T_{i_{k-1}} (F_{i_k})$.

By similar arguments, one can also show that $\pi^\imath \TT_{i_1} \cdots \TT_{i_{k-1}} (B_{i_k'})$ is given by $\T_{i_1} \cdots \T_{i_{k-1}} (F_{i_k'})$. The proof is completed.
\end{proof}

\section{Polynomial Poisson algebras and braid group symmetries}\label{sec:ppa}

Recall from Section~\ref{sec:psk} that the semi-classical limit of $\Ui_{\A'}$ is the Poisson algebra $\C[K^\perp \backslash G^*]$.  In this section, we study the braid group symmetries on such Poisson algebras obtained by taking the specialization of relative braid group symmetries on $\U_{\A'}$. When the underlying Satake diagram is quasi-split, we provide a geometric characterization and explicit formulas for those braid group symmetries on the underlying Poisson manifolds.

\subsection{The polynomial Poisson algebra}

Retain the notations as before. Recall the ring $\Z'\subset\mathbb{C}$ and the ring homomorphism $\A'\rightarrow \Z'$, $q\mapsto 1$, from Section \ref{sec:proj}. We set 
\[
\P=\Z'\otimes_{\A'}\Ui_{\A'}
\]
to be the commutative $\Z'$-algebra carrying the Poisson bracket defined as \eqref{eq:poi}. Let $\P^0$ be the subalgebra generated by $\overline{K_\mu}$ for $\mu\in Y^\imath$. Then $\P^0$ is canonically isomorphic to $\mathbb{Z}'[Y^\imath]$, the group algebra of $Y^\imath$ over $\mathbb{Z}'$.

The following proposition follows immediately from Theorem \ref{thm:PBW}. 

\begin{proposition}\label{prop:poly}
    $\P$ is a polynomial ring over $\P^0$ with generators $\ov{B_{\beta}}, \ov{F_\gamma}, \ov{E_\gamma},$ for $\beta\in \cR^+(\bbw_0),\gamma \in \cR^+_\bullet$. In particular, the degree of $\P$ over $\P^0$ is $l(w_0)+l(w_\bullet)$.
\end{proposition}

For $i\in\I_\bullet$, by Proposition \ref{prop:Tj} the symmetry $T_i$ preserves the integral form $\Ui_{\A'}$ and hence induces a Poisson automorphism on $\P$, which we denote by $\overline{T_i}$. For $i\in\wI$, by Theorem \ref{thm:intt} the symmetry $\TT_i$ preserves $\Ui_{\A'}$ and hence induces a Poisson automorphism $\overline{\TT_i}$ on $\P$. 

Recall from Section~\ref{sec:rbga} that $W_\bullet$ is the Weyl group associated to $\bI$ and $W^\circ$ is the relative Weyl group. Let $\mathrm{Br}(W_\bullet),\mathrm{Br}(W^\circ)$ be the associated braid groups.

\begin{proposition}\label{cor:rbg}
    There exists a braid group action of $\mathrm{Br}(W_\bullet)\rtimes\mathrm{Br}(W^\circ)$ on $\P$ as automorphisms of Poisson algebras generated by $\overline{T_j}$ ($j\in\I_\bullet$) and $\overline{\TT_i}$ ($i\in \wItau$).
\end{proposition}

\begin{proof}
By \cite[Theorem 9.10]{WZ23}, there is an action of $\mathrm{Br}(W_\bullet)\rtimes\mathrm{Br}(W^\circ)$ on $\Ui$ given by $s_j\mapsto T_j,\bs_i\mapsto \TT_i$ for $j\in\I_\bullet,i\in \wItau$; see also Proposition~\ref{prop:WZ}. The desired statement follows by taking specialization $q\mapsto 1$.
\end{proof}

Recall the dual Poisson-Lie group $G^*$ and the Poisson homogeneous space $\X=K^\perp\backslash G^*$ in Section \ref{sec:plg}. Let $pr:G^*\rightarrow \X$ be the canonical projection. 
By Proposition \ref{prop:iphi}, the $\Z'$-algebra $\P$ provides an integral model for the Poisson manifold $\X$, that is, one has the canonical Poisson algebra automorphism 
\begin{equation}\label{eq:bc}
\C\otimes_{\mathbb{Z}'}\P\cong \mathbb{C}[\X].
\end{equation}

For $i\in \I_\bullet$ (resp. $i\in \wI$), let $(\sigma_i^\imath)^*:\mathbb{C}[\X]\rightarrow\mathbb{C}[\X]$ be the Poisson algebra automorphism obtained by the base change of $\overline{T_i}$ (resp. $\overline{\TT_i}$) under the isomorphism \eqref{eq:bc}. Let 
\begin{align}\label{eq:isigma}
\sigma_i^\imath:\X\rightarrow \X
\end{align}
be the Poisson morphism induced by $(\sigma_i^\imath)^*$ (recall from \cite[Proposition 3.2]{So24} that $\X$ is an affine space). The following theorem follows immediately from Proposition \ref{cor:rbg}.

\begin{theorem}\label{thm:rbs}
    There exists a braid group action of $\mathrm{Br}(W_\bullet)\rtimes\mathrm{Br}(W^\circ)$ on $\X$ as Poisson automorphisms generated by $\sigma_i^\imath$ $(i\in\I_\bullet\cup\I_{\circ,\tau})$.
\end{theorem}


\subsection{The geometric description of $\sigma_i^\imath$}\label{sec:gds}

In this subsection, we give a geometric description of the braid group symmetries $\sigma_i^{\imath}$ established in Theorem~\ref{thm:rbs}.
We first provide a description for $\sigma_i^\imath,i\in \bI$. Recall the actions $\sigma_i:G^*\rightarrow G^*$ for $i\in\I$ from \eqref{eq:gbg}. 



\begin{proposition}\label{prop:ib}
    For $i\in \I_\bullet$, the map $\sigma_i^{\imath}$ is the unique Poisson isomorphism such that the diagram
    \begin{equation}\label{eq:sib}
        \begin{tikzcd}
      &  G^* \arrow[r,"\sigma_{i}"] \arrow[d,"pr"] &  G^* \arrow[d,"pr"] \\ & \X \arrow[r,"\sigma_i^\imath"] & \X
    \end{tikzcd}
    \end{equation}
    commutes. In particular, one has $\sigma_i(K^\perp)=K^\perp$.
\end{proposition}

\begin{proof}
By Proposition \ref{prop:Tj}, we have the following commutative diagram
\begin{equation} 
        \begin{tikzcd}
      &  \U_{\A'} \arrow[r,"T_i"]  &  \U_{\A'} 
      \\ 
      & \Ui_{\A'} \arrow[r,"T_i"]\arrow[u] & \Ui_{\A'}\arrow[u]
    \end{tikzcd}
    \end{equation}
The desired \eqref{eq:sib} follows by taking the $q\rightarrow 1$ limit of this commutative diagram. The uniqueness is obvious.
\end{proof}

\emph{In the rest of this section, we assume that the underlying Satake diagram is quasi-split. i.e., $\I_\bullet=\emptyset$.}

Let $\Theta: G^*\rightarrow G^*$ be the map defined by $(b_1,b_2)\mapsto (\theta(b_2)^{-1},\theta(b_1)^{-1})$. This map is well-defined since $\theta(B^+)=B^-$. The map $\Theta$ appeared and was shown to be Poisson in \cite[Section 7]{Boa01} for type AI and in \cite[Example 5.11]{Xu03} for all split types. We will show that $\Theta$ is Poisson for all quasi-split types in Proposition~\ref{prop:spg}.

Let $\psi: \X=K^\perp\backslash G^*\rightarrow G^*$ be the map given by 
\begin{align}\label{eq:psi}
\psi: K^\perp g\mapsto \Theta(g)g,\qquad \text{ for } g\in G^*. 
\end{align}

Following \cite[Definition~5.1]{Xu03}, a \emph{symmetric Poisson group} is a pair $(H,\Psi)$ where $H$ is a Poisson-Lie group and $\Psi:H\rightarrow H$ is a group anti-morphism and a Poisson involution. For a symmetric Poisson group $(H,\Psi)$, it is known (see \cite[Theorem~5.9]{Xu03}) that the identity component of the $\Psi$-fixed locus $H^\Psi=\{g\in H\mid \Psi(g)=g\}$ is a Poisson homogeneous space of $H$, where the $H$-action is given by $g\cdot h=\Psi(h)gh$, for $g\in H^\Psi$ and $h\in H$.

\begin{proposition}\label{prop:spg}
    The pair $(G^*,\Theta)$ is a symmetric Poisson group. Moreover, the map $\psi:\X\rightarrow G^*$ is an isomorphism as Poisson homogeneous spaces onto the identity component of the $\Theta$-fixed locus of $G^*$.
\end{proposition}

\begin{proof}
    The map $\Theta$ is clearly a group anti-morphism, and $\Theta^2=\Id$. To show that $(G^*,\Theta)$ is a Poisson symmetric group, it remains to show that $\Theta$ is Poisson. By \eqref{eq:ty} and \eqref{eq:isg} (recall that we are assuming $c_i=-1)$, the comorphism $\Theta^*:\mathbb{C}[G^*]\rightarrow \mathbb{C}[G^*]$ is the specialization of the algebra involution on $\U_\A$, defined by $E_i\mapsto F_{\tau i}$, $F_{\tau i}\mapsto E_i$, and $K_i\mapsto K_{\tau i}^{-1}$, for $i\in\I$. Therefore $\Theta$ is Poisson.

    Let $Q\subset \{g\in G^*\mid \Theta(g)=g\}$ be the identity component of the $\Theta$-fixed locus. Note that the stablizer of the identity element
    \[
    \text{Stab}_e(G^*)=\{(b_1,b_2)\in G^*\mid b_1=\theta(b_2)\}\cong H'\times U^-,\quad \text{where $H'=\cap_{i\in \I}\text{ ker }\alpha_i\alpha_{\tau i}^{-1}\subset H$,}
    \]
     is connected. Hence $\text{Stab}_e(G^*)=K^\perp$ by the definition (cf. \S \ref{sec:psk}). Since $\psi$ is $G^*$-equivariant and $Q$ is a homogeneous space, we conclude that $\psi$ is an isomorphism onto $Q$.
\end{proof}



\begin{theorem}\label{thm:iw}
Assume that $\I_\bullet=\emptyset$. For $i\in \wItau$, $\sigma_i^{\imath}$ is the unique Poisson automorphism on $K^\perp\backslash G^*$ such that the following diagram commutes
    \begin{equation}
        \begin{tikzcd}
            & \X \arrow[r,"\psi"] \arrow[d,"\sigma_i^\imath"'] & G^* \arrow[d,"\sigma_{\bs_i}"] \\ & \X \arrow[r,"\psi"] & G^*    
        \end{tikzcd}
    \end{equation}
\end{theorem}

\begin{proof}
Let $Q\subset G^*$ be the identity component of the $\Theta$-fixed locus, which is the image of $\psi$ thanks to Proposition \ref{prop:spg}. It follows from the direct computation that
\begin{equation}\label{eq:qhu}
Q=\{(\theta(u)^{-1}t^{-1},tu)\mid t\in H^{\theta 0},u\in U^-\}\cong H^{\theta0}\times U^-,
\end{equation}
where $H^{\theta0}\subset H^\theta$ is the identity component of the $\theta$-fixed locus of $H$. We next recall the Poisson structure on $Q$.

For $q\in Q$, since $\Theta(q)=q$, $\Theta$ induces a linear map on $T_qG^*$ of order 2. Let $T_qG^*=T_qQ\oplus V_q$, where $V_q$ is the subspace consisting of $(-1)$-eigenvectors of $\Theta_*$. Since $\Theta_*$ preserves $\pi\mid_ {G^*}$, one has $\pi_{G^*}(q)=\pi_Q(q)+\pi'(q)$, where $\pi_Q(q)\in \wedge^2 T_qQ$ and $\pi'(q)\in\wedge^2V_q$. Let $\pi_Q\in \Gamma(\wedge^2TQ)$ be such that $\pi_{G^*}\mid _Q=\pi_Q+\pi'$, where $\pi'\in \Gamma(\wedge^2V_Q)$. By \cite[Theorem~5.9]{Xu03}, $(Q,\pi_Q)$ is a Poisson manifold, and the map $\psi:(\X,\pi_\X)\rightarrow (Q,2\pi_Q)$ is Poisson.

We next show that $\sigma_{\bs_i}\circ \Theta=\Theta\circ\sigma_{\bs_i}$ as maps on $G^*$. We prove the equality on the quantum level. Let $\omega:\U_\A\rightarrow\U_\A$ be the algebra automorphism given by $\omega(E_i)=F_{i}$, $\omega(F_i)=E_{i},$ and $\omega(K_i)=K_{i}^{-1}$, 
for $i\in\I$. Let $\tilde{\tau}:\U_\A\rightarrow \U_\A$ be the algebra involution defined by $\tilde{\tau}(E_i)=E_{\tau i}$, $\tilde{\tau}(F_i)=F_{\tau i}$, and $\tilde{\tau}(K_i)=K_{\tau i}$, for $i\in\I$. It is direct to check that $\tilde{\tau}\circ T_i=T_{\tau i}\circ \tilde{\tau}$ and $\omega\circ T_i=T_i\circ \omega$, for $i\in\I$. Therefore we have $T_{\bs_i}\circ(\omega\tilde{\tau})=(\omega\tilde{\tau})\circ T_{\bs_i}$. Then the equality $\sigma_{\bs_i}\circ \Theta=\Theta\circ\sigma_{\bs_i}$ follows from the specialization.

Therefore $\sigma_{\bs_i}$ preserves the $\Theta$-fixed locus of $G^*$. Since $\sigma_{\bs_i}(e)=e$ and $\sigma_{\bs_i}$ is a homeomorphism, we conclude that $\sigma_{\bs_i}$ preserves $Q$. Moreover, since $\sigma_{\bs_i}$ is Poisson, by the construction of $\pi_Q$, we conclude that $\sigma_ {\bs_i}\mid_{Q}:(Q,\pi_Q)\rightarrow (Q,\pi_Q)$ is Poisson.

It suffices to show that the diagram

\begin{equation}\label{eq:geo0}
        \begin{tikzcd}
            & \X \arrow[r,"\psi"] \arrow[d,"\sigma_i^\imath"'] & Q \arrow[d,"\sigma_{\bs_i}\mid_{Q}"] \\ & \X \arrow[r,"\psi"] & Q    
        \end{tikzcd}
    \end{equation}
commutes. Under the isomorphism \eqref{eq:qhu}, we have $\C[Q]\cong\C[H^{\theta0}]\otimes\mathbb{C}[U^-]$. We have $\mathbb{C}[H^{\theta0}]\cong\mathbb{C}[\alpha_i^{\pm1}\mid i\in\I\setminus\wItau]$. 
Then the comorphism 
$$\psi^*:\C[Q]\rightarrow\mathbb{C}[\X]$$
is the Poisson algebra isomorphism given by $\psi^*(\alpha_i^{\pm1})=(\overline{K_iK_{\tau i}^{-1}})^{\pm1}$ for $i\in\I\backslash\wItau$, and $\psi^*(\chi_i^-)=\overline{B_i}$ for $i\in\I$. 

In order to prove that the diagram \eqref{eq:geo0} commutes, it suffices to show that
\begin{equation}\label{eq:fc}
f(\sigma_{\bs_i}\circ\psi(x))=f(\psi\circ\sigma_i^\imath(x)),
\end{equation}
for any $x\in\X$, and $f\in\C[Q]$ equals $\alpha_i$ ($i\in\I\backslash \wItau$) and $\chi_j^-$ ($j\in \I$). By \cite[Proposition~3.2]{So24}, we may assume that $x\in\X$ is represented by $(t^{-1},tu)\in G^*$, where $u\in U^-$, $t\in H$ with $\alpha_i(t)=1$ for all $i\in \wItau$. Then 
\[
\psi(x)=\big((\theta u)^{-1}(\theta t)^{-1}t^{-1},(\theta t)tu\big).
\]
Let us write $u=u_iu^{(i)}$, where $u_i\in U^-\cap \dot{\bs}_iU^+\dot{\bs}_i^{-1}$ and $u^{(i)}=U^-\cap \dot{\bs}_iU^-\dot{\bs}_i^{-1}$. By taking a reduced expression of $w_0$ starting with $\bs_i$ and applying \eqref{eq:gbg}, we have
\begin{align}\label{eq:geo1}
\sigma_{\bs_i}(\psi(x))=(u_+ \cdot \dot{\bs}_i^{-1}(\theta t)^{-1}t^{-1}\dot{\bs}_i,\dot{\bs}_i^{-1}(\theta t)t\dot{\bs}_i\cdot \dot{\bs}_i^{-1}u^{(i)}\dot{\bs}_iu_i')\in Q,
\end{align}
for some $u_i'\in U^-\cap \dot{\bs}_iU^+\dot{\bs}_i^{-1}, u_+\in U^+$. 

We now separate the proof of \eqref{eq:fc} into the following three cases (1)-(3).
\begin{itemize}
\item[(1)] $f=\alpha_j, j\in\I\backslash\wItau$. By \eqref{eq:geo1} and \eqref{eq:isg}, both sides of \eqref{eq:fc} equal $\alpha_j(\dot{\bs}_i^{-1} (\theta t)^{-1} t^{-1}\dot{\bs}_i)$ as desired.

\item[(2)] $f=\chi^-_j,j\notin \{i,\tau i\},j\in \I$. On one hand, by \eqref{eq:isg} and \eqref{eq:geo1}, we have
\[
\chi_j^-(\sigma_{\bs_i}(\psi(x)))=\chi_j^-(\dot{\bs}_i^{-1}u^{(i)}\bs_i u_i')=\chi_j^-(\dot{\bs}_i^{-1}u^{(i)}\bs_i)=\overline{T_{\bs_i}(F_j)}(t^{-1},tu),
\]
where the second equality follows since $u_i'\in U^-\cap \dot{\bs}_iU^+\dot{\bs}_i^{-1}$.

On the other hand, we have
\[
\chi_j^-(\psi\circ\sigma_i^\imath(x))=\overline{\bT_i(B_j)}(x).
\]
By Proposition \ref{prop:piT} we have $\bT_i(B_j)=T_i(F_j)+R$, where $R\in \U_{\A'}\U_{\A'}^{+,>}+\U_{\A'}\U_{\A'}^{0,>}$. Here $\U_{\A'}^{+,>}=\U_{\A'}\cap\text{ker }\epsilon $ and $\U_{\A'}^{0,>}=\U_{\A'}\cap \text{ker }\pi^0$ (cf. \S \ref{sec:proj}). Thus, $\overline{R}\big((t^{-1},tu)\big)=0$. Therefore we have
\[
\chi_j^-(\psi\circ\sigma_i^\imath(x))=\overline{T_i(F_j)}((t^{-1},tu))=\chi_j^-(\sigma_{\bs_i}\circ\psi(x)).
\]

Hence, \eqref{eq:fc} holds for $f=\chi_j^-$ when $j\not\in \{i,\tau i\}$.

\item[(3)] Finally, we prove \eqref{eq:fc} for $f=\chi_i^-$. The proof of \eqref{eq:fc}  for $f=\chi_{\tau i}^-$ is similar and omitted. Note that $T_{\bs_i}(F_i)=q^{n_{i}/2}K_{\tau_i(i)}^{-1}E_{\tau_i(i)}$, for some $n_{i}\in\mathbb{Z}$ (cf. \S \ref{sec:qgp}). 

On the one hand, we have
\begin{align*}
\chi_i^-(\sigma_{\bs_i}\circ\psi(x))&=\overline{T_{\bs_i}(F_i)}(((\theta u)^{-1}(\theta t)^{-1}t^{-1},(\theta t)tu))\\
&=\alpha_{\tau_i(i)}(\theta(t)t)\chi_{\tau_i(i)}^+((\theta u)^{-1})\\
&=\alpha_{\tau_i(i)}\alpha_{\tau \tau_i(i)}^{-1}(t)\chi^-_{\tau\tau_i(i)}(u).
\end{align*}

On the other hand, by Proposition \ref{prop:WZ} we have
\begin{align*}
    \chi_i^-(\psi\circ\sigma_i^\imath(x))&=\overline{\TT_i(B_i)}(x)\\
    &=\overline{B_{\tau\tau_i(i)}K_{\tau\tau_i(i)}K^{-1}_{\tau_i(i)}}(x)\\
    &=\alpha_{\tau_i(i)}\alpha_{\tau \tau_i(i)}^{-1}(t)\chi^-_{\tau\tau_i(i)}(u).
\end{align*}
Hence \eqref{eq:fc} is proved for $f=\chi_i^-$. Similarly one can prove \eqref{eq:fc} for $f=\chi_{\tau i}^-$.
\end{itemize}
We complete the proof.
\end{proof}

\subsection{Explicit actions of $\sigma_i^{\imath}$ on $K^\perp\backslash G^*$ }\label{sec:braidex}
In this subsection, we assume that the underlying Satake diagram $(\I,\tau)$ is quasi-split type. i.e., $\bI=\emptyset$. Recall from \eqref{eq:psi} and \eqref{eq:qhu} the isomorphism $\psi:\X\overset{\sim}{\longrightarrow} Q \cong H^{\theta 0} \times U^-$. By the proof of Theorem~\ref{thm:iw}, the automorphism $\sigma_i^{\imath}$ on $\X$ coincides with the automorphism $\sigma_{\bs_i}\mid_Q$ on $Q$. Using the above isomorphisms, $\sigma_{\bs_i}\mid_Q$ induces a map on $H^{\theta 0} \times U^-$, which will be also denoted by $\sigma_i^{\imath}$. 

Note that when $\bI=\emptyset$ any vertex $i\in \I$ satisfies exactly one of the following three conditions: (i) $a_{i,\tau i}=2$; (ii) $a_{i,\tau i}=0$; (iii) $a_{i,\tau i}=-1$. We describe the explicit actions of $\sigma_i^{\imath}$ on $H^{\theta 0} \times U^-$ for all three cases.

$\boxed{\text{(i) Split type } a_{i,\tau i}=2}$

    Let us consider a Satake diagram $(\I,\Id)$ of split type where $\tau=\Id$. In this setting, one has $\bs_i=s_i, H^{\theta 0}=\{1\}$ and then an isomorphism $U^-\cong Q, u\mapsto (\theta(u^{-1}),u)$.  Fix a reduced expression of $w_0=\bs_{i_1}\cdots \bs_{i_l}$ in $W^\circ$ with $i_1=i$. Then, for any $u\in U^-$, one can write $u=y_i(a) u^{(i)}$ for $u^{(i)}\in U^-\cap \dot{\bs}_i U^- \dot{\bs}_i^{-1}$ and $a=\chi_i^-(u)$. Using \eqref{eq:gbg}, the induced map $\sigma_i^\imath: U^- \rightarrow U^-$ is given by
    \begin{align}
    \sigma_i^\imath:u\mapsto   \dot{\bs}_i^{-1} u^{(i)} \dot{\bs}_i y_i(a).
    \end{align}
    Equivalently, the formula of $\sigma^\imath_i$ can be written as 
    \[
    \sigma^\imath_i: u\mapsto \varphi_i(\begin{pmatrix}
        \chi_i^-(u) & -1 \\ 1 & 
    \end{pmatrix})u
    \varphi_i(\begin{pmatrix}
       \chi_i^-(u)    & 1 \\ -1 & 
    \end{pmatrix}).
    \]
    In particular, for type AI (i.e., split type A), $U^-$ is the space of unipotent lower-triangular matrices. Under the isomorphism $U^-\rightarrow U^+, x\mapsto (x^T)^{-1}$, the Poisson automorphism $\sigma_i^{\imath}$ coincides with the braid group action on the Dubrovin–Ugaglia Poisson bracket (cf. \cite[(F.13)]{Du94}).
    For type CI (i.e., split type C), the associated Poisson brackets and braid group actions were computed via the R-matrix realization of the $\imath$quantum group in \cite{ZLZ25}. It is interesting to compare $\sigma_i^{\imath}$ with their braid group symmetries.


$\boxed{\text{(ii) Diagonal type } a_{i,\tau i}=0}$ 

Let us consider a quasi-split Satake diagram $(\I,\tau)$ and $i\in \I$ such that $a_{i,\tau i}=0$.
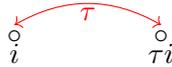
\begin{align*}
\begin{tikzpicture}[baseline=0,scale=1.5]
		\node  at (-0.65,0) {$\circ$};
		\node  at (0.65,0) {$\circ$};
		\node at (-0.65,-0.15) {$i$};
		\node at (0.65,-0.15) {$\tau i$};
        \draw[bend left,<->,red] (-0.65,0.1) to (0.65,0.1);
        \node at (0,0.2) {$\textcolor{red}{\tau}$};
	\end{tikzpicture}
\end{align*}

Set $c_i=c_{\tau i}=-1$. In this setting, $\bs_i=s_i s_{\tau i}$. Fix a reduced expression of $w_0=\bs_{i_1}\cdots \bs_{i_l}$ in $W^\circ$ with $i_1=i$. Then, for any $u\in U^-$, one can write $u=y_i(a) y_{\tau i}(b) u^{(i)}$ such that $u^{(i)}\in U^-\cap \dot{\bs}_i U^- \dot{\bs}_i^{-1}$. By \eqref{eq:gbg}, the induced map $\sigma_i^\imath:H^{\theta 0} \times U^- \rightarrow H^{\theta 0} \times U^-$ is given by
\begin{align}
\sigma_i^\imath:\big((\theta t)t,u\big)\mapsto \big(\dot{\bs}_i^{-1} \;(\theta t)t \; \dot{\bs}_i, \dot{\bs}_i^{-1} u^{(i)} \dot{\bs}_i y_i(\lambda b) y_{\tau i}(\lambda^{-1} a)\big),
\end{align}
where $\lambda=\alpha_i(t)\alpha^{-1}_{\tau i}(t)$.

In particular, let us consider the Satake diagram $(\widetilde{\I} ,\tau)$ is of diagonal type, where $\widetilde{\I}=\{i,i'|i\in \I\}$ and $\tau:i\leftrightarrow i'$. Let $G$ be the adjoint group assoicated to $\I$. Recall the notations $\widetilde{G}=G\times G,\widetilde{K}^\perp\subset \widetilde{G}$ from Example~\ref{ex:diagonal}. The above formula shows that, under the isomorphism $\widetilde{K}^\perp\backslash \widetilde{G}\overset{\sim}{\longrightarrow} G^* $, $\sigma_i^\imath$ on $\widetilde{K}^\perp\backslash \widetilde{G}$ are identified with $\sigma_i$ on $G^*$. 

$\boxed{\text{(iii) Quasi-split type } a_{i,\tau i}=-1}$

Let $(\I,\tau)$ be a quasi-split Satake diagram with and $i\in \I$ such that $a_{i,\tau i}=-1$.
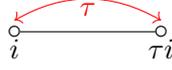
\begin{align*}
\begin{tikzpicture}[baseline=0,scale=1.5]
		\node  at (-0.65,0) {$\circ$};
		\node  at (0.65,0) {$\circ$};
		\draw[-] (-0.6,0) to (0.6, 0);
		\node at (-0.65,-0.15) {$i$};
		\node at (0.65,-0.15) {$\tau i$};
        \draw[bend left,<->,red] (-0.65,0.1) to (0.65,0.1);
        \node at (0,0.2) {$\textcolor{red}{\tau}$};
	\end{tikzpicture}
\end{align*}

Set $c_i=c_{\tau i}=-1$ and fix $\wItau$ such that $i\in \wItau$. In this setting, $\bs_i=s_i s_{\tau i} s_i$. Fix a reduced expression of $w_0=\bs_{i_1}\cdots \bs_{i_l}$ in $W^\circ$ such that $i_1=i$. Then, for any $u\in U^-$, one can write $u=y_{i}(a) y_{\alpha_i+\alpha_{\tau i}}(b) y_{\tau i}(c) u^{(i)}$ where $y_{\alpha_i+\alpha_{\tau i}}(b)=\dot{s}_i y_{\tau i}(b) \dot{s}_i^{-1}$ such that $u^{(i)}\in U^-\cap \dot{\bs}_i U^- \dot{\bs}_i^{-1}$. By \eqref{eq:gbg}, the induced map $\sigma_i^\imath:H^{\theta 0} \times U^- \rightarrow H^{\theta 0} \times U^-$ is given by
\begin{align}
\sigma_i^\imath:\big((\theta t)t,u\big)\mapsto \big((\theta t)t, \dot{\bs}_i^{-1} u^{(i)} \dot{\bs}_i y_{i}(\lambda^{-1}a) y_{\alpha_i+\alpha_{\tau i}}(b) y_{\tau i}(\lambda c)  \big),
\end{align}
where $\lambda=\alpha_i(t)\alpha^{-1}_{\tau i}(t)$.

\section{Examples of the Poisson algebra $\P$}\label{sec:exa}

In this section, we explicitly describe in various examples the Poisson bracket for the Poisson algebra $\P$, using the polynomial generators in Proposition~\ref{prop:poly}. The parameters $c_i$ are all taken to be $-1$. 



\subsection{$\text{Type AI}_{2}$}
Recall the Satake diagram of type AI$_2$ from Table~\ref{table:rktwo}. In this case we have $W^\circ=W,\bs_i=s_i$ and $Y^\imath=0$. Recall the following $\imath$Serre relations from \cite{Let02,Ko14}
\[
\big\{ B_i,\{B_i,B_j\}_q\big\}_{q^{-1}}=-q^{-1} (q^{1/2}+q^{-1/2})^2 B_j,\qquad 1\leq i\neq j\leq 2.
\]
We take the reduced expression $w_0=s_1 s_2 s_1$. Let $b_i=\ov{B_i}$ and $b_{12}=\ov{\TT_1(B_2)}$. By Proposition~\ref{prop:poly}, $\P=\Z'[b_1,b_{12},b_2]$. Using the above $\imath$Serre relations and the formula $\TT_1(B_2)=\frac{\{B_2,B_1\}_q}{q^{1/2}+q^{-1/2}}$ in the proof of Lemma~\ref{lem:braid-rk2}, the Poisson bracket on $\P$ is given by
\begin{align}
\begin{split}
\{b_1,b_2\}&=-b_1 b_2-2b_{12},
\\
\{b_1,b_{12}\}&=b_1 b_{12} +2 b_2,
\\
\{b_{12},b_2\}&=b_2 b_{12} +2 b_1.
\end{split}
\end{align}
It is easy to see that $\P$ is naturally isomorphic to Dubrovin-Ugaglia Poisson algebra \cite{Du94,Uga99}.

\subsection{$\text{Type AIV}_{2}$}
The Satake daigram of type AIV$_2$ is given by
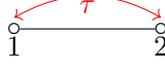
\begin{align*}
\begin{tikzpicture}[baseline=0,scale=1.5]
		\node  at (-0.65,0) {$\circ$};
		\node  at (0.65,0) {$\circ$};
		\draw[-] (-0.6,0) to (0.6, 0);
		\node at (-0.65,-0.15) {$1$};
		\node at (0.65,-0.15) {$2$};
        \draw[bend left,<->,red] (-0.65,0.1) to (0.65,0.1);
        \node at (0,0.2) {$\textcolor{red}{\tau}$};
	\end{tikzpicture}
\end{align*}
In this case we have $W^\circ=\langle \bs_1\rangle$, $\bs_1=s_1s_2s_1$, and $Y^\imath=\mathbb{Z}[\alpha_1-\alpha_2]$. We have the following relations from \cite{Let02,Ko14}
\begin{align*}
    \big\{ B_i,\{B_i,B_j\}_q\big\}_{q^{-1}}=(q^{1/2}+q^{-1/2})^2(1+q^{-2})(q^{3/2}B_ik_i^{-1}+q^{3/2}k_iB_i),\qquad 1\leq i\neq j\leq 2,
    \\
    k_ik_j=1,\quad k_iB_i=q^{-3}B_ik_i,\qquad 1\leq i\neq j\leq 2.
\end{align*}
Following the formula in \eqref{eq:A1}, we have $B_{\beta_1}=B_1$, $B_{\beta_2}=\frac{\{B_2,B_1\}_q}{q^{1/2}+q^{-1/2}}$, and $B_{\beta_3}=B_2$. Let $b_i=\overline{B_{\beta_i}}$, for $1\le i\le 3$, and $k=\overline{k_1}$. By Proposition~\ref{prop:poly}, $\mathcal{P}=\mathbb{Z}'[k^{\pm1},b_1,b_2,b_3]$. The Poisson bracket on $\P$ is given by
\begin{align*}
         \{k,b_1\}&=-3kb_1,  &\{k,b_2\}=0, \\
         \{k,b_3\}&=3kb_3, & \{b_1,b_2\}=4b_1(k+k^{-1})+b_1^3b_3+b_1b_2,\\
         \{b_1,b_3\}&=-2b_2-b_1b_3, &\{b_2,b_3\}=4b_3(k+k^{-1})-b_2b_3.
\end{align*}

\subsection{$\text{Type AIII}_{3}$}
Recall the Satake diagram of type AIII$_3$ from Table~\ref{table:rktwo}. In this case we have $\bs_1=s_1 s_3, \bs_2=s_2$. Recall the following $\imath$Serre relations from \cite{Let02,Ko14} for $j=1,3$
\begin{align*}
\big\{ B_2,\{B_2,B_j\}_q\big\}_{q^{-1}}&=-q^{-1} (q^{1/2}+q^{-1/2})^2 B_j,
\\
\big\{ B_j,\{B_j,B_2\}_q\big\}_{q^{-1}}&=0,
\end{align*}
Fix the reduced expression $w_0=\bs_1 \bs_2 \bs_1 \bs_2$. Let $b_i=\ov{B_i},i=1,2,3$ and $b_{132}=\ov{\TT_1(B_2)},b_{12}=\ov{\TT_1\TT_2(B_3)},b_{23}=\ov{\TT_1\TT_2(B_1)},k=\ov{k_1}$. By \cite[Theorem 7.13]{WZ23}, we have for $j=1,3$
\[
\TT_1\TT_2(B_{\tau j})=\TT_2^{-1}(B_j)=\frac{\{B_2,B_j\}_q}{q^{1/2}+q^{-1/2}}.
\]
By Proposition~\ref{prop:poly}, $\P=\Z'[k^{\pm1},b_1,b_3, b_{132},b_{23},b_{12},b_2]$. Recall from Lemma~\ref{lem:braid-rk2} that $\TT_1(B_2)=\frac{\big\{\{B_2,B_1\}_q,B_3\big\}_q}{(q^{1/2}+q^{-1/2})^2}+B_2=\frac{\big\{\{B_2,B_3\}_q,B_1\big\}_q}{(q^{1/2}+q^{-1/2})^2}+B_2$. Using the above $\imath$Serre relations and the formulas of $\TT_1(B_2),\TT_2^{-1}(B_j),j=1,3$, the Poisson bracket on $\P$ is given by
\begin{align}
\begin{split}
&\{k,b_1\}=-2 k b_1, \qquad \{k,b_2\}=0,\qquad \{k,b_3\}=2kb_3,
\\
&\{b_1,b_3\}=2(k-k^{-1})=\{b_{12},b_{23}\},
\qquad \{b_1,b_{12}\}=b_1 b_{12},\qquad \{b_3, b_{23}\}=b_3 b_{23},
\\
&\{b_1,b_2\}=-b_1 b_2-2b_{12},\qquad \{b_3,b_2\}=-b_3 b_2-2 b_{23},
\\
&\{b_1,b_{132}\}=b_1 b_{132}+2b_{12}(k+k^{-1}-1),\qquad \{b_1,b_{23}\}=-b_1 b_{23}+2b_2-2b_{132},
\\
&\{b_3,b_{132}\}=b_3 b_{132}+2b_{23}(k+k^{-1}-1),\qquad \{b_3,b_{12}\}=-b_3 b_{12}+2b_2-2b_{132},
\\
&\{b_{132},b_{23}\}=b_{132}b_{23}+2b_3 k^{-1},\qquad \{b_{132},b_{12}\}=b_{132}b_{12}+2 b_1 k,
\\
&\{b_{12},b_2\}=b_{12} b_2 + 2 b_1,\qquad \{b_{23},b_2\}=b_{23} b_2+2b_3,
\\
&\{b_{132},b_2\}=2b_1 b_3-2b_{12}b_{23}.
\end{split}
\end{align}

\subsection{$\text{Type CI}_{2}$}
Recall the Satake diagram of type CI$_2$ from Table~\ref{table:rktwo}. In this case we have $W^\circ=W,\bs_i=s_i$, and $Y^\imath=0$. Recall the following $\imath$Serre relations from \cite{Let02,Ko14}
\begin{align*}
\big\{ B_2,\{B_2,B_1\}_{q^2}\big\}_{q^{-2}}&=-q^{-2} (q+q^{-1})^2 B_1,
\\
\Big\{ B_1, \big\{B_1,\{B_1,B_2\}_{q^2}\big\}_1\Big\}_{q^{-2}}
&=-q^{-1} (q^{1/2}+q^{-1/2})^2[2]^2 \{B_1, B_2\}_1.
\end{align*}
We take the reduced expression $w_0=s_1 s_2 s_1 s_2$. Let $b_i=\ov{B_i}$, $b_{12}=\ov{\TT_1(B_2)}$ and $b_{121}=\ov{\TT_1\TT_2(B_1)}$. By Proposition~\ref{prop:poly}, $\P=\Z'[b_1,b_{12},b_{121},b_2]$. Using the above $\imath$Serre relations and the formula of $\TT_2(B_1)$ in the proof of Lemma~\ref{lem:braid-rk2}, the Poisson bracket on $\P$ is given by
\begin{align}
\begin{split}
\{b_2,b_1\}&=2b_2 b_1+2b_{121},
\qquad\qquad \qquad \quad \{b_{121},b_2\}=2b_2 b_{121}+2b_1,
\\
\{b_1,b_{12}\}&=2b_1 b_{12}+4b_1b_2+6b_{121},
\qquad 
\{b_1,b_{121}\}=2b_2-2b_{12},
\\
\{b_{12},b_{121}\}&=2b_{12}b_{121}+2b_1,
\qquad \qquad \qquad \quad 
\{b_{12},b_2\}=2b_{121}^2-2b_1^2.
\end{split}
\end{align}

\newpage
\appendix

\section{Rank 1 root vectors in $\Ui_{\A'}$}\label{app:A}

Recall the root vectors $B_\beta$ from \eqref{def:Bbeta} and the integral form $\Ui_{\A'}$ from \eqref{def:UiA}.
In this section, we provide explicit formulas for the root vectors $B_\beta$ for $\beta\in \cR^+(\bbw_0)$ for each rank 1 Satake diagram, using formulas in \cite[Section 3]{LYZ24}. We verify that these root vectors $B_\beta$ lie in the integral form $\Ui_{\A'}$ for each rank 1 Satake diagram. We recall from \cite{WZ23} the list of rank 1 Satake diagrams in Table~\ref{table:localSatake}.


\begin{table}[H]
\caption{Rank 1 Satake diagrams and local datum}
\label{table:localSatake}
\resizebox{5.5 in}{!}{%
\begin{tabular}{| c | c | c |}
\hline
Type & Satake diagram  & $\bs_i$
\\
\hline
\begin{tikzpicture}[baseline=0]
\node at (0, -0.15) {AI$_1$};
\end{tikzpicture}

&	
    \begin{tikzpicture}[baseline=0]
		\node  at (0,0) {$\circ$};
		\node  at (0,-.3) {\small 1};
	\end{tikzpicture}
& $\bs_1=s_1$
\\
\hline
\begin{tikzpicture}[baseline=0]
\node at (0, -0.15) {AII$_3$};
\end{tikzpicture}
&
   \begin{tikzpicture}[baseline=0]
		\node at (0,0) {$\bullet$};
		\draw (0.1, 0) to (0.4,0);
		\node  at (0.5,0) {$\circ$};
		\draw (0.6, 0) to (0.9,0);
		\node at (1,0) {$\bullet$};
		\node at (0,-.3) {\small 1};
		\node  at (0.5,-.3) {\small 2};
		\node at (1,-.3) {\small 3};
	\end{tikzpicture}
& $\bs_2=s_{2132}$
\\
\hline
\begin{tikzpicture}[baseline=0]
\node at (0, -0.2) {AIII$_{11}$};
\end{tikzpicture}
 &
\begin{tikzpicture}[baseline = 6] 
		\node at (-0.5,0) {$\circ$};
		\node at (0.5,0) {$\circ$};
		\draw[bend left, <->] (-0.5, 0.2) to (0.5, 0.2);
		\node at (-0.5,-0.3) {\small 1};
		\node at (0.5,-0.3) {\small 2};
	\end{tikzpicture}
& $\bs_1=s_1 s_2$
\\
\hline
\begin{tikzpicture}[baseline=0]
\node at (0, -0.2) {AIV$_n$, n$\geq$2};
\end{tikzpicture}
&
\begin{tikzpicture}	[baseline=6]
		\node at (-0.5,0) {$\circ$};
		\draw[-] (-0.4,0) to (-0.1, 0);
		\node  at (0,0) {$\bullet$};
		\node at (2,0) {$\bullet$};
		\node at (2.5,0) {$\circ$};
		\draw[-] (0.1, 0) to (0.5,0);
		\draw[dashed] (0.5,0) to (1.4,0);
		\draw[-] (1.6,0)  to (1.9,0);
		\draw[-] (2.1,0) to (2.4,0);
		\draw[bend left, <->] (-0.5, 0.2) to (2.5, 0.2);
		\node at (-0.5,-.3) {\small 1};
		\node  at (0,-.3) {\small 2};
		\node at (2.5,-.3) {\small n};
	\end{tikzpicture}
& $\bs_1=s_{1 \cdots  n \cdots 1} $
\\
\hline
\begin{tikzpicture}[baseline=0]
\node at (0, -0.2) {BII, n$\ge$ 2};
\end{tikzpicture} &
    	\begin{tikzpicture}[baseline=0, scale=1.5]
		\node at (1.05,0) {$\circ$};
		\node at (1.5,0) {$\bullet$};
		\draw[-] (1.1,0)  to (1.4,0);
		\draw[-] (1.4,0) to (1.9, 0);
		\draw[dashed] (1.9,0) to (2.7,0);
		\draw[-] (2.7,0) to (2.9, 0);
		\node at (3,0) {$\bullet$};
		\draw[-implies, double equal sign distance]  (3.05, 0) to (3.55, 0);
		\node at (3.6,0) {$\bullet$};
		\node at (1,-.2) {\small 1};
		\node at (1.5,-.2) {\small 2};
		\node at (3.6,-.2) {\small n};
	\end{tikzpicture}	
& $ \bs_1=s_{1\cdots n \cdots 1}$
\\
\hline
\begin{tikzpicture}[baseline=0]
\node at (0, -0.15) {CII, n$\ge$3};
\end{tikzpicture}
&
		\begin{tikzpicture}[baseline=6]
		\draw (0.6, 0.15) to (0.9, 0.15);
		\node  at (0.5,0.15) {$\bullet$};
		\node at (1,0.15) {$\circ$};
		\node at (1.5,0.15) {$\bullet$};
		\draw[-] (1.1,0.15)  to (1.4,0.15);
		\draw[-] (1.4,0.15) to (1.9, 0.15);
		\draw (1.9, 0.15) to (2.1, 0.15);
		\draw[dashed] (2.1,0.15) to (2.7,0.15);
		\draw[-] (2.7,0.15) to (2.9, 0.15);
		\node at (3,0.15) {$\bullet$};
		\draw[implies-, double equal sign distance]  (3.1, 0.15) to (3.7, 0.15);
		\node at (3.8,0.15) {$\bullet$};
		\node  at (0.5,-0.15) {\small 1};
		\node at (1,-0.15) {\small 2};
		\node at (3.8,-0.15) {\small n};
	\end{tikzpicture}
&	$\bs_2=s_{2\cdots n\cdots 2 1 2\cdots n\cdots 2} $	
\\
\hline
\begin{tikzpicture}[baseline=0]
\node at (0, -0.05) {DII, n$\ge$4};
\end{tikzpicture}&
	\begin{tikzpicture}[baseline=0]
		\node at (1,0) {$\circ$};
		\node at (1.5,0) {$\bullet$};
		\draw[-] (1.1,0)  to (1.4,0);
		\draw[-] (1.4,0) to (1.9, 0);
		\draw[dashed] (1.9,0) to (2.7,0);
		\draw[-] (2.7,0) to (2.9, 0);
		\node at (3,0) {$\bullet$};
		\node at (3.5, 0.4) {$\bullet$};
		\node at (3.5, -0.4) {$\bullet$};
		\draw (3.05, 0.05) to (3.4, 0.39);
		\draw (3.05, -0.05) to (3.4, -0.39);
		\node at (1,-.3) {\small 1};
		\node at (1.5,-.3) {\small 2};
		\node at (3.6, 0.25) {\small n-1};
		\node at (3.5, -0.6) {\small n};
	\end{tikzpicture}		
&
\begin{tikzpicture}[baseline=0]
\node at (0, 0.15) { $\bs_1=$};
\node at (0, -0.35) { $ s_{1 \cdots n-2 \cdot n-1 \cdot n \cdot n-2  \cdots 1}$};
\end{tikzpicture}
\\
\hline
\begin{tikzpicture}[baseline=0]
\node at (0, -0.2) {FII};
\end{tikzpicture}
&
\begin{tikzpicture}[baseline=0][scale=1.5]
	\node at (0,0) {$\bullet$};
	\draw (0.1, 0) to (0.4,0);
	\node at (0.5,0) {$\bullet$};
	\draw[-implies, double equal sign distance]  (0.6, 0) to (1.2,0);
	\node at (1.3,0) {$\bullet$};
	\draw (1.4, 0) to (1.7,0);
	\node at (1.8,0) {$\circ$};
	\node at (0,-.3) {\small 1};
	\node at (0.5,-.3) {\small 2};
	\node at (1.3,-.3) {\small 3};
	\node at (1.8,-.3) {\small 4};
\end{tikzpicture}
& $\bs_4=s_{432312343231234}$
\\
\hline
\end{tabular}
}
\newline
\smallskip
\end{table}

Let $(\I=\wI\cup\bI,\tau)$ be a rank $1$ Satake diagram. In the rank 1 case, $\wI=\{i,\tau i\}$ and $\bbw_0=\bs_i$.


\subsection{$\un{\text{Type AI}_1}$}
In this case,
$\cR^+(\bs_1)=\{\alpha_1\}$. Set $B_{\alpha_1}=B_1 \in \Ui_{\A'}$. 

\subsection{$\un{\text{Type AII}_3}$}
In this case,
$\cR^+(\bs_2)=\{\alpha_2,\alpha_1+\alpha_2,\alpha_2+\alpha_3,\alpha_1+\alpha_2+\alpha_3\}$.
Set
\begin{align*}
\begin{split}
B_{\alpha_2}=&B_2,\qquad B_{\alpha_1+\alpha_2}=T_1^{-1}(B_2),
\\
B_{\alpha_2+\alpha_3}=&T_3^{-1}(B_2),
\qquad
B_{\alpha_1+\alpha_2+\alpha_3}= T_{13}^{-1}(B_2).
\end{split}
\end{align*}
Note that $B_\beta,\beta\in \cR^+(\bs_2)$ all have the form $T_w^{-1}(B_i)$ for some $w\in W_\bullet,i\in \wI$. By Proposition~\ref{prop:Tj}, $B_\beta\in \Ui_{\A'}$ for any $\beta\in \cR^+(\bs_2)$.
\subsection{$\un{\text{Type  AIII}_{11}}$}

In this case, $\cR^+(\bs_1)=\{\alpha_1,\alpha_2\}$. We set $B_{\alpha_1}=B_1$, $B_{\alpha_2}=B_2$. Clearly, they lie in $\Ui_{\A'}$.

\subsection{$\un{\text{Type  AIV}_n,n\geq2}$ }

In this case, $\cR^+(\bs_1)=\{\beta_i\mid 1\leq i\leq 2n-1 \}$, where
\begin{align*}
\beta_i=\begin{cases}
\alpha_1+\alpha_2+\cdots+\alpha_i, & \text{ for }1\leq i\leq n,
\\
\alpha_n+\alpha_{n-1}+\cdots+\alpha_{2n-i+1}, &\text{ for }n+1\leq i\leq 2n-1.
\end{cases}
\end{align*}
We define
\begin{align}\label{eq:A1}
\begin{split}
B_{\beta_i}=&T_{2\cdots i}^{-1}(B_1),\qquad 1\leq i\leq n-1,
\\
B_{\beta_n}=&\frac{[B_n,T_{2\cdots(n-1)}^{-1}(B_1)]_q}{q^{1/2}(q-q^{-1})} 
=\frac{\big\{B_n,T_{2\cdots(n-1)}^{-1}(B_1)\big\}_q}{q^{1/2}+q^{-1/2}},
\\
B_{\beta_{n+i}}=&T_{(n-1)\cdots(n-i+1) }^{-1} (B_n),\qquad 1\leq i\leq n-1.
\end{split}
\end{align}
By Lemma~\ref{lem:qcomi} and Proposition~\ref{prop:Tj} $B_{\beta_i}\in \Ui_\A$ for any $1\leq i\leq 2n-1$.

\subsection{$\un{\text{Type  DII},n\geq4}$ }
In this case, $\cR^+(\bs_1)=\{\beta_i\mid 1\leq i\leq 2n-2\}$ where
\begin{align*}
\beta_i=
\begin{cases}
\alpha_1+\alpha_2+\cdots+\alpha_i, & \text{ for }1\leq i\leq n-1,
\\
\alpha_1+\alpha_2+\cdots+\alpha_{n-2}+\alpha_n , &\text{ for } i=n,
\\
\alpha_1+\cdots+\alpha_{2n-i-1}+2\alpha_{2n-i}+\cdots+2\alpha_{n-2} +\alpha_{n-1}+\alpha_n, &\text{ for } n+1\leq i\leq 2n-2.
\end{cases}
\end{align*}
We define
\begin{align}
\begin{split}
&B_{\beta_i}=T_{2\cdots i}^{-1}(B_1),\qquad 1\leq i\leq n-1,
\\
&B_{\beta_{n}}= T_{2\cdots (n-2)\cdot n}^{-1}(B_1),
\\
&B_{\beta_{n+1}}=T_{2\cdots n}^{-1}(B_1),
\\
&B_{\beta_{n+i}}=T_{2\cdots n\cdot (n-2)\cdot (n-3)\cdots (n-i)}^{-1}(B_1),\qquad 2\leq i\leq n-2.
\end{split}
\end{align}
For the same reason as type AII$_3$, $B_{\beta_{i}}$ for $ 1\leq i \leq 2n-2$ all lie in $\Ui_{\A'}$.

\subsection{$\un{\text{Type BII}, n\geq2}$} 

In this case, $\cR^+(\bs_1)=\{\beta_i\mid 1\leq i\leq 2n-1\}$, where
\begin{align}
\beta_i=\begin{cases}
\alpha_1+\alpha_2+\cdots +\alpha_i,& \text{ for }1\leq i\leq n,
\\
\alpha_1+\cdots +\alpha_{2n-i}+2\alpha_{2n+1-i} +\cdots +2\alpha_n  ,&\text{ for }n+1\leq i\leq 2n-1.
\end{cases}
\end{align}
We define 
\begin{align*}
&B_{\beta_i}=T_{2\cdots i}^{-1}(B_1),\qquad 1\leq i\leq n-1,
\\
&B_{\beta_n}=\frac{\big[F_n,T_{2\cdots(n-1)}^{-1}(B_1)\big]_{q^2}}{q^2-q^{-2}}=\frac{[F_n,B_{\beta_{n-1}}]_{q^2}}{q(q^2-q^{-2})}
=\frac{\{F_n,B_{\beta_{n-1}}\}_{q^2}}{q+q^{-1}},
\\
&B_{\beta_{n+i}}=T_{2\cdots n\cdots(n-i+1)}^{-1}(B_1),\qquad 1\leq i\leq n-1.
\end{align*}
By Lemma~\ref{lem:qcomi} and Proposition~\ref{prop:Tj}. it is clear that $B_{\beta_{i}}\in \Ui_{\A'}$ for $1\le i\le 2n-1$.

\subsection{$\un{\text{Type CII}, n\geq3}$} 

In this case, $\cR^+(\bs_2)=\{\beta_i\mid 1\leq i\leq 4n-5\}$, where 
\begin{align*}
	\beta_i= \begin{cases}  \alpha_2+\alpha_3+\cdots+\alpha_{i+1}, & \text { for } 1  \leq i \leq n-2, 
		\\ 
		2\alpha_2+2 \alpha_3+\cdots+2 \alpha_i+\alpha_{i+1}, & \text { for } i =n-1,
		\\
		\alpha_2+\cdots+\alpha_{2n-i-1}+2\alpha_{2n-i}+\cdots+2\alpha_{n-1}+\alpha_n, & \text { for } n \leq i \leq 2 n-3 ,
		 \\
		\alpha_1+2 \alpha_2+\cdots+2 \alpha_{n-1}+\alpha_n, & \text { for } i=2 n-2 ,
  \\ \alpha_1+\alpha_2+\cdots+\alpha_{i-2 n+3}, & \text { for } 2n-1 \leq i \leq 3n-4, 
		  \\ 
		2 \alpha_1+\cdots+2 \alpha_{n-1}+\alpha_n, & \text { for } i=3n-3,
		   \\
\alpha_1+\cdots+\alpha_{4n-3-i}+2\alpha_{4n-2-i}+\cdots+2\alpha_{n-1}+\alpha_n, & \text{ for } 3n-2\leq i\leq4n-5,
\end{cases}
\end{align*}
We define
\begin{align} 
\begin{split}
&B_{\beta_i}= T_{3\cdots (i+1)}^{-1}(B_2),\qquad 1\leq i\leq n-2,
 \\
&B_{\beta_{n-1}} =\frac{1}{[2]q(q-q^{-1})^2} \Big[\big[F_n, B_{\beta_{n-2}}\big]_{q^2},B_{\beta_{n-2}}\Big]
=\frac{\Big\{\big\{F_n, B_{\beta_{n-2}}\big\}_{q^2},B_{\beta_{n-2}}\Big\}}{[2](q^{1/2}+q^{-1/2})^2},
\\
&B_{\beta_{n+i}}= T_{3\cdots n\cdots(n-i)}^{-1}(B_2), \qquad 0\leq i\leq n-3,
    \\
&B_{\beta_{2n-2}}=\frac{1}{q^{1/2}(q-q^{-1})}\big[T_1^{-1}(B_2), B_{\beta_{2n-3}}\big]_q
=\frac{\big\{T_1^{-1}(B_2), T_{3\cdots n\cdots 3}^{-1}(B_{2})\big\}_q}{q^{1/2}+q^{-1/2}},
    \\
&B_{\beta_{i+2n-1}}= T_{13\cdots (i+2)}^{-1}(B_2),\qquad 0\leq i\leq n-3,
    \\
&B_{\beta_{3n-3}} =\frac{1}{[2]q(q-q^{-1})^2}\big[[F_n,B_{\beta_{3n-2}}]_{q^2},B_{\beta_{3n-2}}\big]
=\frac{\Big\{\big\{F_n,B_{\beta_{3n-2}}\big\}_{q^2},B_{\beta_{3n-2}}\Big\}}{[2](q^{1/2}+q^{-1/2})^2},
    \\
&B_{\beta_{i+3n-2}} = T_{13\cdots n\cdots(n-i)}^{-1}(B_2),
    \qquad 0\leq i\leq n-3.
\end{split}
\end{align}
By Lemma~\ref{lem:qcomi} and Proposition~\ref{prop:Tj}, it is clear that $B_{\beta_{i}}\in \Ui_{\A'}$ for $1\le i\le 4n-5$.

\subsection{$\un{\text{Type FII}}$} 
Let $\beta_{i},1\le i\le 15$ be the set of all roots in $\cR^+(\bs_4)$; their labeling is obtained from the reduced expression of $\bs_4$ in Table~\ref{table:localSatake}.
We define 
\begin{align}
\begin{split}
&B_{\beta_1}=B_4,\qquad B_{\beta_{2}}=T_{3}^{-1}(B_4),\qquad B_{\beta_{3}}=\frac{1}{[2]q(q-q^{-1})^2}\big[[F_2,B_{\beta_2}]_{q^2},B_{\beta_2}\big],
\\
&B_{\beta_{4}}=T_{32}^{-1}(B_4),\qquad
B_{\beta_{5}}=T_{1}^{-1}(B_{\beta_3}),\qquad
B_{\beta_{6}}=\frac{1}{q(q^2-q^{-2})}[F_2,B_{\beta_5}]_{q^2},\qquad
\\
&B_{\beta_{7}}=T_{321}^{-1}(B_4),\qquad
B_{\beta_{8}}=\frac{1}{q^{1/2}(q-q^{-1})}[T_{323}^{-1}(B_4),B_{\beta_7}]_q,\qquad
B_{\beta_{9}}=T_{323}^{-1}(B_4),\qquad
\\
&B_{\beta_{10}}=\frac{1}{[2]q(q-q^{-1})^2}\big[[F_1,B_{\beta_9}]_{q^2},B_{\beta_9}\big],
\qquad
B_{\beta_{11}}=T_{3213}^{-1}(B_4),
\\
&B_{\beta_{12}}=\frac{1}{q(q^2-q^{-2})}[F_2,B_{\beta_{10}}]_{q^2},
\qquad
B_{\beta_{13}}=\frac{1}{q(q^2-q^{-2})}[F_1,B_{\beta_{12}}]_{q^2},
\\
&B_{\beta_{14}}=T_{32312}^{-1}(B_4),\qquad
B_{\beta_{15}}=T_{323123}^{-1}(B_4).
\end{split}
\end{align}
By Proposition~\ref{prop:Tj}, $B_{\beta_{i}}\in \Ui_{\A'}$ for $i\in \{1,2,4,5,7,9,11,14,15\}$.
We show that $B_{\beta_i}\in \U_{\A'}$ for the remaining $i$. Indeed, we have
\begin{align*}
&B_{\beta_{3}}=\frac{\big\{\{F_2,B_{\beta_2}\}_{q^2},B_{\beta_2}\big\}}{[2](q^{1/2}+q^{-1/2})^2},
\qquad 
B_{\beta_{6}}=\frac{\{F_2,B_{\beta_5}\}_{q^2}}{(q^{1/2}+q^{-1/2})^2 },
\qquad
B_{\beta_{8}}=\frac{\{T_{323}^{-1}(B_4),B_{\beta_7}\}_q}{q^{1/2}+q^{-1/2}},
\\
&B_{\beta_{10}}=\frac{\big\{\{F_1,B_{\beta_9}\}_{q^2},B_{\beta_9}\big\}}{[2](q^{1/2}+q^{-1/2})^2},
\qquad
B_{\beta_{12}}=\frac{\{F_2,B_{\beta_{10}}\}_{q^2}}{q+q^{-1}},
\qquad
B_{\beta_{13}}=\frac{\{F_1,B_{\beta_{12}}\}_{q^2}}{q+q^{-1}}.
\end{align*}
By Lemma~\ref{lem:qcomi} and Proposition~\ref{prop:Tj}, $B_{\beta_{i}}\in \Ui_{\A'}$ for $i\in \{3,6,8,10,12,13\}$ as desired.


\begin{thebibliography}{AWW18}
\bibitem[AV22]{AV22} A.~Appel and B.~Vlaar,
\textit{Universal $k$-matrices for quantum Kac-Moody algebras}, Represent. Theory {\bf 26} (2022), 764--824. 

\bibitem[BG17]{BG17} A.~Berenstein and J.~Greenstein,
{\em Double canonical bases}, Adv. Math. 316 (2017), 381–468.



\bibitem[BS22]{BS22}
H. Bao and J. Song,
{\em Symmetric subgroup schemes, Frobenius splittings, and quantum symmetric pairs},
\href{https://arxiv.org/abs/2212.13426}{arXiv:2212.13426}

\bibitem[Boa01]{Boa01} 
P. P. Boalch, {\em Stokes matrices, Poisson Lie groups and Frobenius manifolds}, 
Invent. Math. {\bf 146} (2001), no. 3, 479–506.

\bibitem[Boa02]{Boa02}
Philip P. Boalch, \emph{G-bundles, isomonodromy and quantum Weyl groups}, 
International Mathematics Research Notices, 2002(22): 1129–1166, 2002. 

\bibitem[BW18a]{BW18a} H. Bao and W. Wang,
{\em  A new approach to Kazhdan-Lusztig theory  of type $B$ via quantum symmetric pairs}, Ast\'erisque {\bf 402}, 2018, vii+134pp. 

\bibitem[BW18b]{BW18b} H. Bao and W. Wang,
{\em Canonical bases arising from quantum symmetric pairs}, Invent. Math. {\bf 213}(3) (2018), 1099--1177. 


\bibitem[DCK90]{DK90} C. De Concini and V. G. Kac, {\em Representations of quantum groups at roots of 1}, in: Progress in Mathematics {\bf 92}, Birkhäuser Boston, 1990, 471–506. 

\bibitem[DCKP92]{DCKP92} C. De Concini, V. G. Kac, and C. Procesi, {\em Quantum coadjoint action}, J. Amer. Math. Soc. {\bf 5} (1992), no. 1, 151–189.

\bibitem[DCKP93]{DCKP93} C. De Concini, V. G. Kac, and C. Procesi, {\em Some quantum analogs of solvable Lie groups}, \href{https://arxiv.org/abs/hep-th/9308138}{arXiv:hep-th/9308138}.

\bibitem[DCP93]{DCP93} C. De Concini and C. Procesi, {\em Quantum groups}, D-modules, Representation Theory, and Quantum Groups: Lectures given at the 2nd Session of the Centro Internazionale Matematico Estivo (C.I.M.E.) held in Venezia, Italy, June 12--20, 1992


\bibitem[Do20]{Do20} L.~Dobson,
{\em Braid group actions for quantum symmetric pairs of type AIII/AIV}, J. Algebra {\bf564} (2020), 151--198.

\bibitem[DK19]{DK19} L.~Dobson and S.~Kolb,
{\em Factorisation of quasi K-matrices for quantum symmetric pairs}, Selecta Math. (N.S.) {\bf 25} (2019), Paper No. 63, 53 pp.

\bibitem[Du94]{Du94} B.~Dubrovin, 
{\em Geometry of 2d topological field theories}, 
\href{https://arxiv.org/abs/hep-th/9407018}{arXiv:hep-th/9407018}, 1994.

		


\bibitem[Jan96]{Jan96} J.C.~Jantzen,
{\em Lectures on quantum groups}, Grad. Studies in Math. {\bf 6}, Amer. Math. Soc., Providence (1996).

\bibitem[Ko14]{Ko14} S.~Kolb,
{\em Quantum symmetric Kac-Moody pairs}, Adv. Math. {\bf 267} (2014), 395--469.


\bibitem[KP11]{KP11} S.~Kolb and J.~Pellegrini,
		{\em Braid group actions on coideal subalgebras of quantized enveloping algebras}, J. Algebra {\bf 336} (2011), 395--416.

\bibitem[KY21]{KY21} S.~Kolb, M.~Yakimov,
{\em Defining relations of quantum symmetric pair coideal subalgebras}, Forum of Mathematics, Sigma, 2021.


\bibitem[Let02]{Let02} G.~Letzter,
{\em Coideal subalgebras and quantum symmetric pairs},
New directions in Hopf algebras (Cambridge), MSRI publications, {\bf 43}  (2002), Cambridge Univ. Press, 117--166.

\bibitem[Let19]{Let19} G.~Letzter,
 {\em Cartan subalgebras for quantum symmetric pair coideals}, Represent. Theory
 {\bf23} (2019), 99--153.

\bibitem[LP25]{LP25} M.~Lu and X.~Pan
{\em Dual canonical bases of quantum groups and $\imath$quantum groups},
\href{https://arxiv.org/abs/2504.19073}{arXiv:2504.19073}

\bibitem[LW21]{LW21a} M.~Lu and W.~Wang,
{\em Hall algebras and quantum symmetric pairs II: reflection functors}, Commun. Math. Phys. {\bf 381} (2021), 799--855.

\bibitem[LW22a]{LW19a} M.~Lu and W.~Wang,
{\em Hall algebras and quantum symmetric pairs I: foundations}, Proc. London Math. Soc. (3) {\bf 124} (2022), 1--82.

\bibitem[LW22b]{LW21} M. Lu and W. Wang,
{\em Braid group symmetries on quasi-split $\imath$quantum groups via $\imath$Hall algebras}, Selecta Math. {\bf 28} (2022), Article number 84.


\bibitem[Lus90a]{Lus90a} G.~Lusztig,
\textit{Finite-dimensional Hopf algebras arising from quantized universal enveloping algebra}, J. Amer. Math. Soc. {\bf 3} (1990), 257--296.

\bibitem[Lus90b]{Lus90b} G.~Lusztig,
\textit{Canonical bases arising from quantized enveloping algebras}, J. Amer. Math. Soc. {\bf 3} (1990), 447--498.



\bibitem[Lus93]{Lus93} G. Lusztig,
		{\em Introduction to quantum groups}, Modern Birkh\"auser Classics, Reprint of the 1993 Edition, Birkh\"auser, Boston, 2010.

\bibitem[Lus03]{Lus03}
G. Lusztig,
{\em Hecke algebras with unequal parameters},
CRM Monograph Series {\bf 18}, Amer. Math. Soc., Providence, RI, 2003, \href{https://arxiv.org/abs/math/0208154}{arXiv:0208154v2}


\bibitem[LYZ24]{LYZ24} M.~Lu, R.~Yang, and W.~Zhang
{\em PBW basis for $\imath$quantum groups}, \href{https://arxiv.org/abs/2407.13127}{arXiv:2407.13127}

\bibitem[MR08]{MR08} A. Molev, E. Ragoucy,
{\em Symmetries and invariants of Twisted Quantum Algebra and Associated Poisson Algebra}, Rev. Math. Phys. {\bf 20} (2008), 173-198.

\bibitem[Qin20]{Qin20} F. Qin,
{\em Dual canonical bases and quantum cluster algebras}, \href{https://arxiv.org/abs/2003.13674}{arXiv:2003.13674}

\bibitem[Sh22a]{Sh22a} L.~Shen,
{\em Cluster Nature of Quantum Groups}, \href{https://arxiv.org/abs/2209.06258}{arXiv:2209.06258}

\bibitem[Sh22b]{Sh22b} L.~Shen, 
{\em Duals of semisimple Poisson--Lie groups and cluster theory of moduli spaces of $G$-local systems}, 
\emph{Int. Math. Res. Not.}, Vol.~2022, No.~18, pp.~14295--14318.

\bibitem[So24a]{So23} J.~Song,
{\em Cluster realisation of $\imath$quantum groups of type AI}, 
Proc. London Math. Soc. {\bf 129} (2024), no. 4, Paper No. e12636, 48 pp.

\bibitem[So24b]{So24} J.~Song,
{\em Quantum duality principle and quantum symmetric pairs}, \href{https://arxiv.org/abs/2403.05167}{arXiv:2403.05167}

\bibitem[Spr09]{Spr09} T. A.~Springer,
{\em Linear Algebraic Groups}, 2nd ed., Modern Birkh{\"a}user Classics, Birkh{\"a}user Boston, Inc., Boston, MA, 2009.


\bibitem[Su23]{Su23} M.~Sugawara,
{\em Quantum dilogarithm identities arising from the product formula for the universal R-matrix of quantum affine algebras}, Publ. RIMS Kyoto Univ {\bf 59} (2023), no. 4, 769--819.

\bibitem[Uga99]{Uga99} M.~Ugaglia, 
{\em On a Poisson structure on the space of Stokes matrices}, International Mathematics Research Notices {\bf 9} (1999), 473–-493.

\bibitem[Wa23]{Wa23} W.~Wang,
{\em Quantum symmetric pairs}, In Proceedings of ICM, Vol. {\bf 4} (2023), 3080--3102. EMS Press, Berlin.

\bibitem[WZ23]{WZ23} W. Wang and W. Zhang,
	{\em An intrinsic approach to relative braid group symmetries on $\imath$quantum groups}, Proc. London Math. Soc. {\bf 127} (2023), 1338--1423.

\bibitem[WZ25]{WZ25} W.~Wang and W.~Zhang,
{\em Relative braid group symmetries on modified $\imath$quantum groups and their modules}, in preparation.

\bibitem[Xu03]{Xu03} P. Xu,
    {\em Dirac submanifolds and Poisson involutions}, 
    Ann. Sci. Éc. Norm. Supér. {\bf 36} (2003), 403--430.



\bibitem[ZLZ25]{ZLZ25}
Y. Zhang, H. Lin, and H. Zhang,
{\em RTT presentation of coideal subalgebra of quantized enveloping algebra of type\,CI},
\href{https://arxiv.org/abs/2501.13305v1}{arXiv:2501.13305}

	
\end{thebibliography}
\end{document}